 \documentclass[11pt,a4paper,reqno]{amsart}
 \usepackage[T1]{fontenc}
 \usepackage{lmodern}
 \usepackage[margin=27mm,headheight=14pt]{geometry}
 \usepackage{amssymb,mathtools,mathrsfs,booktabs,array,microtype}
 \usepackage{xurl}
 \usepackage[hidelinks]{hyperref}
 \usepackage{bookmark}
 \usepackage{needspace}
 \hypersetup{pdftitle={Almost global existence for quadratic Schrodinger equations in two dimensions},pdfauthor={Xuan Liu},pdfsubject={Quadratic two-dimensional NLS, space-time resonance, endpoint Besov estimates}}
 \numberwithin{equation}{section}
 \newtheorem{theorem}{Theorem}[section]
 \newtheorem{proposition}[theorem]{Proposition}
 \newtheorem{lemma}[theorem]{Lemma}
 
 \theoremstyle{definition}
 
 \theoremstyle{remark}
 \newtheorem{remark}[theorem]{Remark}
 \newcommand{\R}{\mathbb R}
 
 \newcommand{\Z}{\mathcal Z}
 
 \newcommand{\e}{\mathrm e}
 \newcommand{\dd}{\,\mathrm d}
 \newcommand{\supp}{\operatorname{supp}}
 \newcommand{\norm}[1]{\lVert #1\rVert}
 \newcommand{\Rea}{\operatorname{Re}}
 \newcommand{\one}{\mathbf 1}
 \allowdisplaybreaks[2]
 \title[Quadratic Schr\"odinger equations in two dimensions]{Almost global existence for quadratic Schr\"odinger equations in two dimensions}
 \author{Xuan Liu}
 \subjclass[2020]{35Q55, 35B40}
 \keywords{Quadratic Schr\"odinger equation, almost global existence, space-time resonance, normal form, Besov space}
 
\begin{document}

 	\maketitle
 	
 	\begin{abstract}
 		We study the two-dimensional quadratic Schr\"odinger equation
 		$(i\partial_t+\Delta)u=\lambda_1u^2+\lambda_2\overline u^{\,2}$
 		with initial data of size $\varepsilon$
 		in a weighted $H^2$ space. We prove almost global
 		existence, with a lifespan of at least $\exp(c/\varepsilon)$
 		and decay at the optimal linear dispersive rate $t^{-1}$.
 		Our proof is based on the space--time resonance method
 		introduced by Germain, Masmoudi, and Shatah.
 		The main new ingredient is a Besov-type profile norm
 		$\mathcal Z$, which allows us to close the bootstrap argument
 		without controlling the global second spatial moment
 		of the free profile, thereby avoiding the substantial
 		time-dependent losses arising from direct
 		second-moment estimates.
 	\end{abstract}
 	
 	\section{Introduction and main result}\label{sec:intro}
 	
 	We study the Cauchy problem for the two-dimensional nonlinear Schr\"odinger equation (NLS) with two non-gauge-invariant quadratic terms:
 	\begin{equation}\label{eq:pde}
 		\begin{cases}
 			(i\partial_t+\Delta)u=\lambda_1 u^2+\lambda_2\overline u^{\,2},
 			&(t,x)\in[1,\infty)\times\mathbb R^2,\\
 			u(1,x)=u_1(x),
 		\end{cases}
 	\end{equation}
 	where $\lambda_1,\lambda_2\in\mathbb C$ are fixed constants. 
 	We denote the free profile by $f(t)=e^{-it\Delta}u(t)$
 	and take $t=1$ as the initial time for convenience.

 	Equation \eqref{eq:pde} lies at a critical threshold in the small-data theory of NLS. For the  model $(i\partial_t+\Delta)u=\mathcal N_p(u)$ with homogeneity $p>1$, two exponents help describe the small-data dynamics in dimension $d$:
 	\begin{itemize}
 	\item The \emph{Strauss exponent}
 	$p_{\mathrm{St}}(d)
 	=\frac{d+2+\sqrt{d^2+12d+4}}{2d}$,
 	above which small-data global existence
 	and linear scattering can be established
 	by using Strichartz estimates and  contraction mapping argument
 	\cite{Strauss81}.
 
 		\item The \emph{short-range exponent}
 		$p_{\mathrm{sr}}(d)=1+2/d$.
 		If a solution decays at the linear dispersive rate
 		$\|u(t)\|_\infty\lesssim t^{-d/2}$, then
 		the effective potential $|u|^{p-1}$ is integrable
 		in time precisely when $p>p_{\mathrm{sr}}(d)$.
 		For the gauge-invariant nonlinearity
 		$|u|^{p-1}u$, small-data global existence
 		and linear scattering have been established
 		in this range for $d=1,2,3$, under suitable
 		assumptions on the initial data
 		\cite{CazenaveWeissler1992,GinibreOzawaVelo1994,
 			NakanishiOzawa2002}.
 		
 	\end{itemize}
In two dimensions, $p_{\mathrm{sr}}(2)=2$ while
$p_{\mathrm{St}}(2)=1+\sqrt{2}$. Thus quadratic
nonlinearities lie at the short-range threshold
and strictly below the Strauss exponent.
At this power, the algebraic and resonant
structure of the nonlinearity plays an important
role in the long-time behavior of solutions.

For the gauge-invariant nonlinearity
$\lambda|u|u$, $\lambda\in\mathbb R\setminus\{0\}$,
the $L^2$ norm is conserved. Since the nonlinearity
is $L^2$-subcritical, global well-posedness in
$L^2$ follows from the standard contraction argument
and mass conservation.
However, ordinary scattering to nontrivial
free states fails \cite{Barab1984,Strauss1974}.
Instead, modified scattering with a logarithmic
phase correction and the corresponding modified
wave operators have been studied extensively;
see, e.g.,
\cite{CazenaveNaumkin2018,HayashiNaumkin1998,
	HNST,LiuZhang2023,Shan2026}.

The remaining  quadratic powers
$u^2$, $u\overline u$, and $\overline u^2$
are not gauge invariant and exhibit different
space--time resonance structures.
For the interactions $u^2$ and $\overline u^2$,
the corresponding quadratic phases are
\[
\Phi_+(\xi,\eta)=2(\xi-\eta)\cdot\eta,
\qquad
\Phi_-(\xi,\eta)
=|\xi|^2+|\xi-\eta|^2+|\eta|^2.
\]
Their space--time resonant sets
$\{\Phi_\alpha=0,\ \nabla_\eta\Phi_\alpha=0\}$
both reduce to the origin $(\xi,\eta)=(0,0)$. 
Germain--Masmoudi--Shatah introduced the
space--time resonance method in their study
of three-dimensional quadratic Schr\"odinger
equations \cite{GMS3D}. They subsequently
applied this method to the two-dimensional
case \cite{GMS2D}, establishing small-data
global existence and scattering for quadratic
nonlinearities under the assumption that
the nonlinearities carry a first-order
derivative at low frequencies.
Wave operators for $u^2$ and $\overline u^2$
have also been constructed
\cite{MTT,ShimomuraTonegawa}.

 The mixed interaction
 $u\overline u=|u|^2$ has a substantially
 different resonance structure. Its phase is
 \[
 \Phi_{+-}(\xi,\eta)
 =|\xi|^2-|\xi-\eta|^2+|\eta|^2
 =2\xi\cdot\eta,
 \]
 and its space--time resonant set is
 \[
 \mathcal R_{+-}
 =\{(\xi,\eta)\in\mathbb R^2\times\mathbb R^2:
 \xi=0\}.
 \]
 Unlike the previous cases, this set is
 two-dimensional, making the resonant
 interactions considerably more difficult
 to control. For this nonlinearity,
 Shimomura \cite{Shimomura} proved the
 nonexistence of nontrivial asymptotically
 free solutions under suitable assumptions.
 By imposing an additional derivative structure
 on the quadratic nonlinearity, Zheng
 \cite{Zheng2014} established small-data
 global existence and scattering.

Several other works have established global existence
and scattering for two-dimensional quadratic
Schr\"odinger equations under additional derivative
or regularity assumptions; see, e.g.,
\cite{BernalHayashiNaumkin,Cohn,Delort,HayashiKato,
	HayashiNaumkin2001,HayashiNaumkin2002,
	HayashiNaumkin2012}.

In this paper, we consider the quadratic nonlinearities
$u^2$ and $\overline u^{\,2}$ without imposing
any derivative structure. We prove almost global
existence for small initial data in a weighted
Sobolev space, together with the optimal
$t^{-1}$ dispersive decay.

 	\subsection{Main result}
 	
 	We first introduce the spaces used in the statement
 	of our main theorem. Let
 	$\Sigma^2=\{g\in H^2(\mathbb R^2):|x|^2g\in L^2\}$
 	be equipped with the norm
 	$\|g\|_{\Sigma^2}=\|g\|_{H^2}+\||x|^2g\|_2$.
 	
 	We also use an endpoint profile space defined
 	through a smooth spatial annular partition
 	$\{\rho_j\}_{j\ge0}$, specified in
 	Subsection~\ref{subsec:profile-space}. Set
 	\begin{equation}\label{eq:Z-def}
 		\|g\|_{\mathcal Z} := \sum_{j\ge0} 2^j \|\rho_j g\|_2,
 		\qquad \mathcal Z := \{g \in L^2_{\mathrm{loc}}(\mathbb R^2) : \|g\|_{\mathcal Z} < \infty\}.
 	\end{equation}
 	This norm is equivalent to
 	$\|\widehat g\|_{B^1_{2,1}(\mathbb R^2_\xi)}$
 	and satisfies
 	\begin{equation}\label{eq:basic-intro}
 		\|g\|_1 + \|g\|_2 + \|xg\|_2 \lesssim \|g\|_{\mathcal Z},
 		\qquad
 		\|\widehat g\|_\infty + |t| \|e^{it\Delta}g\|_\infty \lesssim \|g\|_{\mathcal Z}.
 	\end{equation}
 	
 	We can now state our main result.
 	
 	\begin{theorem}\label{thm:main}
 		Fix $\lambda_1, \lambda_2 \in \mathbb{C}$. There exist constants $c, C, \varepsilon_0 > 0$ such that for any initial data $u_1 \in \Sigma^2$ satisfying $0 < \varepsilon := \|u_1\|_{\Sigma^2} \le \varepsilon_0$, the Cauchy problem \eqref{eq:pde} has a unique solution
 		\[
 		u \in C([1,T_\varepsilon];\Sigma^2)\cap C^1([1,T_\varepsilon];L^2), \qquad T_\varepsilon = \exp(c/\varepsilon),
 		\]
 		and the solution satisfies the uniform bound
 		\begin{equation}\label{eq:main}
 			\sup_{1\le t\le T_\varepsilon} \left( \|u(t)\|_{H^2} + \|f(t)\|_{\mathcal Z} + t\|u(t)\|_\infty + \|\widehat f(t)\|_\infty \right) \le C\varepsilon.
 		\end{equation}
 	\end{theorem}

\begin{remark}\label{rem:gauge}
	The conclusion of Theorem~\ref{thm:main}
	remains valid if an additional gauge-invariant
	quadratic term is included, namely, for
	\[
	(i\partial_t+\Delta)u
	=\lambda_1u^2+\lambda_2\overline u^{\,2}
	+\lambda_3|u|u,
	\qquad \lambda_3\in\mathbb C.
	\]
	The proof is given in Appendix~\ref{app:gauge}.
\end{remark}

Theorem~\ref{thm:main} extends the almost-global
existence result of Murphy--Pusateri \cite{MP}
for the one-dimensional cubic NLS to the
two-dimensional quadratic setting.
They established almost global existence
for small weighted Sobolev data on a time
interval of length at least
$\exp(c\varepsilon^{-2})$, together with the
optimal $t^{-1/2}$ dispersive decay.

 The role of spatial weights in obtaining
 dispersive decay can be seen from the
 standard linear estimate
 \[
 \|e^{it\Delta}g\|_{L^\infty(\mathbb R^d)}
 \lesssim
 t^{-d/2}\|g\|_{L^1(\mathbb R^d)}.
 \]
 In one dimension, the $L^1$ norm on the
 right-hand side can be bounded using
 only the first spatial moment:
 \[
 \|g\|_{L^1(\mathbb R)}
 \lesssim
 \|g\|_{L^2(\mathbb R)}^{1/2}
 \|xg\|_{L^2(\mathbb R)}^{1/2}.
 \]
 Murphy--Pusateri \cite{MP} obtained a more
 refined linear estimate in terms of the
 Fourier profile:
 \[
 t^{1/2}\|u(t)\|_{L^\infty_x}
 \lesssim
 \|\widehat f(t)\|_{L^\infty_\xi}
 +t^{-1/4}\|xf(t)\|_{L^2_x},
 \]
 where $f(t)=e^{-it\Delta}u(t)$ is the free profile.
 Accordingly, they introduced the bootstrap norm
 \[
 \|u(t)\|_{X(t)}
 =\frac12\left(
 \|\widehat f(t)\|_{L^\infty_\xi}
 +t^{-1/4}\|xf(t)\|_{L^2_x}
 \right).
 \]
 This estimate allows the first spatial moment
 to grow like $t^{1/4}$ while preserving the
 optimal $t^{-1/2}$ decay.
 They closed the bootstrap argument using
 the space--time resonance method.
 
   In two dimensions, however, the first spatial moment is no longer sufficient to bound the $L^1$-norm. Controlling $\|g\|_{L^1(\mathbb R^2)}$ via  the usual weighted H\"older inequalities requires the second spatial moment:
 \[
 \|g\|_{L^1(\mathbb R^2)} \lesssim \|g\|_{L^2(\mathbb R^2)}^{1/2} \||x|^2 g\|_{L^2(\mathbb R^2)}^{1/2}.
 \]

 This presents a difficulty in the
 space--time resonance analysis.
 Indeed, estimating $\||x|^2f(t)\|_2$
 requires taking two derivatives of the
 Fourier profile. When these derivatives
 fall on the oscillatory factor
 $e^{is\Phi(\xi,\eta)}$, they produce
 terms involving
 \[
 s\,\partial_\xi^2\Phi
 \qquad\text{and}\qquad
 s^2(\partial_\xi\Phi)^2,
 \]
 which introduce prohibitive time growth and derivative loss for quadratic nonlinearities without structural cancellation.
 
 To circumvent this obstruction, we introduce the Besov-type
 profile norm $\mathcal Z$. By design, the $\mathcal{Z}$-norm satisfies the sharp linear embedding
 \[
 \|e^{it\Delta}g\|_\infty \lesssim t^{-1}\|g\|_{\mathcal Z},
 \]
 thereby replacing the need for global second-moment bounds $\||x|^2 f(t)\|_2$ with the control of $\|f(t)\|_{\mathcal Z}$. We achieve this by decomposing the profile into spatially localized annular pieces and performing localized second-moment estimates on the resulting output. This localized approach allows us to propagate the $\mathcal{Z}$-norm and close the bootstrap argument without
 controlling $\||x|^2f(t)\|_2$ globally.

 	\subsection{Outline of the proof and key estimates}
 	
 	Our proof is inspired by the space--time resonance method
 	developed by Germain--Masmoudi--Shatah
 	\cite{GMS3D,GMS2D}; see also \cite{GermainSTR,MP}.  
 	The main  new ingredients  are a Besov-type
 	profile norm $\mathcal Z$,
 	a one-sided weighted bilinear estimate, and spatially
 	localized estimates for space-nonresonant interactions.
 	We next describe how these ingredients yield the
 	a priori bound for Theorem~\ref{thm:main}.
 	
 	Set
 	\[
 	A(t)=\|f(t)\|_{\mathcal Z},
 	\qquad
 	X_T=\sup_{1\le s\le T}A(s).
 	\]
 	By \eqref{eq:basic-intro},
 	\[
 	\|u(t)\|_\infty\lesssim t^{-1}A(t).
 	\]
 	Thus our main task is to control the profile norm
 	$A(t)$ over an exponentially long time interval.

Following the space--time resonance framework
of \cite{GMS3D,GMS2D}, we examine the quadratic
phases
\[
\Phi_+(\xi,\eta)=2(\xi-\eta)\cdot\eta,
\qquad
\Phi_-(\xi,\eta)
=|\xi|^2+|\xi-\eta|^2+|\eta|^2.
\]
For both interactions, the joint space--time
resonant set consists only of the origin.
We therefore isolate the low-frequency
interactions using a time-dependent cutoff
at joint-frequency scale $s^{-1/2}$.
The remaining interactions are divided into
time- and space-nonresonant regions, where
we integrate by parts in time and frequency,
respectively.

 	We first consider the low-frequency and
 	space-nonresonant quadratic terms.
 	The principal estimate is
 	\[
 	\begin{aligned}
 		&\big\|e^{-is\Delta}T_{\chi_S}
 		(e^{is\Delta}g,e^{is\Delta}h)\big\|_{\mathcal Z}\ +
 		\sum_{\alpha=\pm1}
 		\big\|e^{-is\Delta}T_{\chi_{\mathrm{lo}}}
 		(e^{i\alpha s\Delta}g,
 		e^{i\alpha s\Delta}h)\big\|_{\mathcal Z} \lesssim
 		s^{-1}\|g\|_{\mathcal Z}\|h\|_{\mathcal Z},
 	\end{aligned}
 	\]
 	where the bilinear multiplier $T_m$ is defined in
 	Subsection~\ref{subsec:notation}, and the cutoffs
 	are specified in Subsection~\ref{subsec:partition}.
 	The low-frequency contribution is controlled
 	using the restriction of both input frequencies
 	to a ball of radius $O(s^{-1/2})$.
 	In two dimensions, the area of this ball
 	provides the required factor $s^{-1}$.

 	For the space-nonresonant quadratic term,
 	we decompose the inputs into spatially localized
 	pieces $g_j=\rho_jg$ and $h_k=\rho_kh$.
 	We prove
 	\begin{equation}\label{E}
 		\begin{aligned}
 			&\big\|e^{-is\Delta}T_{\chi_S}
 			(e^{is\Delta}g_j,e^{is\Delta}h_k)
 			\big\|_{\mathcal Z} \lesssim
 			\frac{2^{j+k}}s\|g_j\|_2\|h_k\|_2,
 		\end{aligned}
 	\end{equation}
 	for each pair $(j,k)$. Summing over $j,k$
 	then yields the desired bilinear estimate.

 	To prove \eqref{E}, we split the joint-frequency
 	region into moderate and high frequencies,
 	which are treated by different methods.
 	At moderate frequencies, the relevant multipliers
 	satisfy uniform Coifman--Meyer estimates after normalization by the relevant spatial scales.
 	Together with frequency localization, these
 	estimates provide control of the $L^2$ norm
 	and the second spatial moment of the output.
 	Applying Lemma~\ref{lem:moment} then yields
 	the desired $\mathcal Z$ estimate.

 	For the high-frequency contributions, we use
 	spatial localization to carry out the
 	space-nonresonant integration by parts.
 	Following the method of
 	Germain--Masmoudi--Shatah \cite{GMS3D,GMS2D},
 	we work with the complete oscillatory phase
 	\[
 	\Theta(\xi,\eta;y,z)
 	=2s(\xi-\eta)\cdot\eta
 	-y\cdot(\xi-\eta)-z\cdot\eta,
 	\]
 	where $y$ and $z$ denote the spatial variables
 	of the two inputs.
 	The spatial localization allows us to control
 	the terms involving $y-z$ in
 	$\nabla_\eta\Theta$ and obtain a lower bound
 	for the phase gradient on the relevant
 	high-frequency regions.
 	Repeated integration by parts in $\eta$
 	then yields summable estimates.

 In the time-nonresonant region, we integrate
 by parts in time using Shatah's normal-form
 method \cite{Shatah1985}. This produces
 quadratic boundary terms and cubic remainders,
 with multipliers containing the inverse phase
 $\Phi_\alpha^{-1}$. The singularity of these
 multipliers near the resonant set is handled
 by a one-sided weighted bilinear estimate:
 \[
 \big\|e^{-is\Delta}T_{m_N}
 (e^{i\alpha s\Delta}g,v)\big\|_{\mathcal Z}
 \lesssim
 s\min(1,N^{-1})
 \|g\|_{\mathcal Z}\|v\|_\infty.
 \]
 Here $m_N$ denotes a dyadically localized
 normal-form multiplier, defined precisely
 in Section~\ref{sec:normal}.
 The key point is that the $\mathcal Z$ norm
 is required of only one input, while the
 other is controlled solely in $L^\infty$.

 	The one-sided estimate is particularly useful
 	for the cubic normal-form remainders.
 	Indeed, the profile equation and the dispersive
 	decay give
 	\[
 	\|e^{i\alpha s\Delta}\partial_s f_\alpha\|_\infty
 	\lesssim
 	(|\lambda_1|+|\lambda_2|)
 	s^{-2}A(s)^2,
 	\]
 	where $\alpha=\pm1$, $f_+=f$, and
 	$f_-=\overline f$.
 	Thus, the differentiated profile can be
 	placed in the unweighted input of the
 	bilinear estimate, without requiring
 	control of its spatial moments.
 	This yields a cubic contribution bounded by
 	$s^{-1}(1+\log s)A(s)^3$.

 	Combining the preceding estimates, we obtain
 	\[
 	X_T\le C_0\varepsilon
 	+C_\lambda(1+\log T)X_T^2
 	+C_\lambda(1+\log T)^2X_T^3.
 	\]
 	A standard bootstrap argument then yields
 	the almost-global existence and decay
 	estimates in Theorem~\ref{thm:main}
 	for $T\le\exp(c/\varepsilon)$.

\subsection{Organization of the paper}

Section~\ref{sec:tools} collects notation and preliminary estimates.
Section~\ref{sec:normal} develops the resonance decomposition and
the normal-form identity. In Section~\ref{sec:reduction}, we state
the main analytic lemmas and deduce Theorem~\ref{thm:main} from
them. Section~\ref{sec:weightedNF} proves the one-sided multiplier
and time normal-form estimates, while Section~\ref{sec:spaceproof}
treats the low-frequency and space-nonresonant terms.
Finally, Appendix~\ref{app:gauge} extends the main result
to equations with an additional gauge-invariant
quadratic term.

 	\section{Notation and useful lemmas}\label{sec:tools}
 	
 	In this section, we collect the notation and preliminary
 	estimates used throughout the paper.
 	Subsection~\ref{subsec:notation} introduces the Fourier
 	conventions and the free Schr\"odinger group.
 	Subsection~\ref{subsec:profile-space} describes the
 	annular profile space and its basic properties.
 	Subsection~\ref{subsec:spatial-localization} establishes
 	spatial localization estimates, while
 	Subsection~\ref{subsec:multiplier-bounds} recalls the
 	bilinear multiplier bounds needed later.
 	
 	\subsection{Notation and the free Schr\"odinger group}
 	\label{subsec:notation}

 For $1\le p\le\infty$, we write
 $\|g\|_p=\|g\|_{L^p(\mathbb R^2)}$,
 unless otherwise specified.
 For nonnegative quantities $A$ and $B$, we write
 $A\lesssim B$ if $A\le CB$, where $C$ is independent
 of the time and localization parameters, and
 $A\approx B$ if both $A\lesssim B$ and $B\lesssim A$.
 For a multiindex $\beta$, $|\beta|$ denotes its order.
 Dyadic numbers $N$ range over $2^{\mathbb Z}$,
 while $j,k\ge0$ denote spatial annular indices.
 
 	We use the unitary Fourier transform in two dimensions:
 	\begin{equation*}
 		\widehat g(\xi)=\frac1{2\pi}\int_{\mathbb R^2}
 		\e^{-ix\cdot\xi}g(x)\,dx,
 		\qquad
 		\widehat{\e^{it\Delta}g}(\xi)
 		=\e^{-it|\xi|^2}\widehat g(\xi).
 	\end{equation*}
 	Unless otherwise stated, all Lebesgue norms are
 	taken over $\mathbb R^2$.
 	
 	For a bilinear symbol $m(\xi,\eta)$, we define
 	\begin{equation*}
 		\widehat{T_m(v,w)}(\xi)
 		=\frac1{2\pi}\int_{\mathbb R^2}m(\xi,\eta)
 		\widehat v(\xi-\eta)\widehat w(\eta)\,d\eta.
 	\end{equation*}
 	
 	The free Schr\"odinger group satisfies
 	\begin{equation*}
 		\e^{it\Delta}g(x)=\frac1{4\pi it}
 		\int_{\mathbb R^2}\e^{i|x-y|^2/(4t)}g(y)\,dy,
 		\qquad
 		\|\e^{it\Delta}g\|_2=\|g\|_2,
 		\quad t\ne0.
 	\end{equation*}
 	In particular,
 	\begin{equation}\label{eq:free-dispersion}
 		\|\e^{it\Delta}g\|_\infty
 		\lesssim |t|^{-1}\|g\|_1.
 	\end{equation}
 	
 	\subsection{The annular profile space}
 	\label{subsec:profile-space}
 	
 	Fix a nonnegative, radial, nonincreasing function
 	$\rho\in C_c^\infty(\mathbb R^2)$, equal to $1$
 	for $|x|\le1$ and vanishing for $|x|\ge2$.
 	Define
 	\begin{equation*}
 		\rho_0(x)=\rho(x),\qquad
 		\rho_j(x)=\rho(2^{-j}x)-\rho(2^{1-j}x),
 		\quad j\ge1.
 	\end{equation*}
 	Then $\sum_{j\ge0}\rho_j=1$, the supports have
 	uniformly finite overlap, and
 	\begin{equation}\label{eq:rhosupport}
 		\supp\rho_0\subset\{|x|\le2\},\qquad
 		\supp\rho_j\subset\{2^{j-1}\le|x|\le2^{j+1}\},
 		\quad j\ge1.
 	\end{equation}
 	
 	Recall the norm $\|\cdot\|_{\mathcal Z}$ defined
 	in \eqref{eq:Z-def}. For $g_j=\rho_jg$, the support
 	properties imply
 	\[
 	\norm{g_j}_1\le C2^j\norm{g_j}_2,
 	\qquad
 	\norm{xg_j}_2\le2^{j+1}\norm{g_j}_2.
 	\]
 	Summing over $j\ge0$, we obtain
 	\begin{equation}\label{eq:embedding-proof}
 		\norm g_1+\norm g_2+\norm{xg}_2
 		\lesssim\norm g_{\Z}.
 	\end{equation}
 	Moreover,
 	\[
 	\|\widehat g\|_\infty\le(2\pi)^{-1}\|g\|_1,
 	\qquad
 	|t|\|\e^{it\Delta}g\|_\infty\lesssim\|g\|_1.
 	\]
 	Together with \eqref{eq:embedding-proof}, these
 	estimates give \eqref{eq:basic-intro}.
 	
 	By Plancherel's theorem and the spatial annular
 	decomposition,
 	\[
 	\|\widehat g\|_{B^1_{2,1}(\mathbb R^2_\xi)}
 	\approx\|g\|_{\mathcal Z}.
 	\]
 	
 	\subsection{Spatial localization and low-frequency gains}
 	\label{subsec:spatial-localization}
 	
 	We call $g_j=\rho_jg$ a \emph{spatial atom}.
 	For every multiindex $\beta$, the support properties
 	of $\rho_j$ give
 	\begin{equation*}
 		\norm{x^\beta g_j}_2
 		\lesssim_\beta2^{j|\beta|}\norm{g_j}_2,
 		\qquad
 		\norm{x^\beta g_j}_1
 		\lesssim_\beta2^{j(|\beta|+1)}\norm{g_j}_2.
 	\end{equation*}
 	
 	Let $P_{\le N}$ be a smooth Fourier cutoff obtained
 	by rescaling a fixed symbol, equal to the identity on
 	$|\xi|\le N$ and supported in $|\xi|\le2N$. By Plancherel's theorem and the
 	Fourier $L^\infty$ bound,
 	\begin{align*}
 		\norm{P_{\le CN}(x^\beta g_j)}_2
 		&\lesssim2^{j|\beta|}\norm{g_j}_2,\\
 		\norm{P_{\le CN}(x^\beta g_j)}_2
 		&\lesssim N\norm{\widehat{x^\beta g_j}}_\infty
 		\lesssim N2^{j(|\beta|+1)}\norm{g_j}_2.
 	\end{align*}
 	Combining these bounds yields
 	\begin{equation}\label{eq:atom-low}
 		\norm{P_{\le CN}(x^\beta g_j)}_2
 		\lesssim_\beta
 		2^{j(|\beta|+1)}\min(2^{-j},N)\norm{g_j}_2.
 	\end{equation}
 	The same estimate holds for
 	$\one_{|\xi|\le CN}\partial_\xi^\beta\widehat g_j$
 	in $L^2_\xi$, since
 	$\partial_\xi^\beta\widehat g_j
 	=(-i)^{|\beta|}\widehat{x^\beta g_j}$.
 	
 	By \eqref{eq:free-dispersion}, we also have
 	\begin{equation}\label{eq:atom-disp}
 		\norm{\e^{\pm is\Delta}g_j}_\infty
 		\lesssim s^{-1}2^j\norm{g_j}_2,
 		\qquad s>0.
 	\end{equation}
 	
 	To estimate the $\mathcal Z$ norm of localized
 	outputs, we use the following moment inequality.
 	
 	\begin{lemma}[Recovering the annular norm from two moments]
 		\label{lem:moment}
 		For every $L\ge1$,
 		\begin{equation*}
 			\norm w_{\Z}
 			\lesssim L\norm w_2
 			+L^{-1}\norm{|x|^2w}_2.
 		\end{equation*}
 	\end{lemma}
 	
 	\begin{proof}
 		For $2^j\le L$, Cauchy--Schwarz and finite overlap give
 		\[
 		\sum_{2^j\le L}2^j\norm{\rho_jw}_2
 		\le\left(\sum_{2^j\le L}2^{2j}\right)^{1/2}
 		\left(\sum_j\norm{\rho_jw}_2^2\right)^{1/2}
 		\lesssim L\norm w_2.
 		\]
 		For $2^j>L$, one has $j\ge1$ and $|x|\approx2^j$
 		on the support of $\rho_j$. Thus
 		\[
 		\sum_{2^j>L}2^j\norm{\rho_jw}_2
 		\lesssim\sum_{2^j>L}2^{-j}\norm{\rho_j|x|^2w}_2
 		\lesssim L^{-1}\norm{|x|^2w}_2.
 		\]
 		Adding the two estimates proves the lemma.
 	\end{proof}
 	
 	In particular,
 	\begin{equation}\label{eq:Zfrom2}
 		\norm g_{\Z}\lesssim\norm g_2+\norm{|x|^2g}_2.
 	\end{equation}
 	The freedom to choose $L$ for each localized
 	output will be important in the subsequent estimates.
 	
 	\subsection{Two multiplier bounds}
 	\label{subsec:multiplier-bounds}
 	
 	We first recall the Coifman--Meyer estimate
 	\begin{equation}\label{eq:CM}
 		\norm{T_m(v,w)}_2\lesssim\norm v_2\norm w_\infty,
 		\qquad
 		\norm{T_m(v,w)}_2\lesssim\norm v_\infty\norm w_2,
 	\end{equation}
 	provided that, for sufficiently many multiindices $\beta$,
 	\begin{equation*}
 		\sup_{(\xi,\eta)\ne(0,0)}
 		\big(|\xi-\eta|^2+|\eta|^2\big)^{|\beta|/2}
 		|\partial_{\xi,\eta}^\beta m(\xi,\eta)|<\infty.
 	\end{equation*}
 	See \cite[Section~5.1]{GMS2D} and
 	\cite[Theorem~1.1]{MPTT}.
 	
 	We shall also use the following compactly supported
 	multiplier estimate, which follows from a standard
 	inverse-Fourier kernel bound.
 	
 	\begin{lemma}[A compactly supported bilinear multiplier estimate]
 		\label{lem:compact}
 		Let $N>0$ and $M>0$. Suppose that $b_N(\xi,\eta)$
 		is supported in
 		\[
 		|\xi-\eta|+|\eta|\le C_*N
 		\]
 		and satisfies
 		\[
 		|\partial_{\xi,\eta}^{\beta}b_N(\xi,\eta)|
 		\le C_\beta M N^{-|\beta|}
 		\]
 		for all multiindices $|\beta|\le 6$.
 		Then
 		\begin{equation*}
 			\|T_{b_N}(v,w)\|_{L^2}
 			\lesssim M\|v\|_{L^2}\|w\|_{L^\infty}.
 		\end{equation*}
 		The implicit constant is independent of $N$ and $M$.
 	\end{lemma}

 	\section{Normal forms and the resonance decomposition}\label{sec:normal}
 	
 	In this section, we derive the profile equation, examine the resonance geometry of the two quadratic interactions, and introduce the frequency decomposition. These yield the normal-form representation and the symbol bounds used later.
 	
 	\subsection{The profile equation and a normal-form identity}\label{subsec:earlyNF}
 	We begin by expressing the nonlinear evolution in terms of the free profile. To treat the two quadratic interactions simultaneously, let
 	\begin{equation*}	f_+=f,\qquad f_-=\overline f,\qquad
 		\lambda_+=\lambda_1,\qquad\lambda_-=\lambda_2.
 	\end{equation*}
 	The conjugate field satisfies $\e^{-is\Delta}f_-=\overline{u(s)}$.
 	We define the associated quadratic phases by
 	\begin{equation*}	\Phi_\alpha(\xi,\eta)
 		=|\xi|^2-\alpha|\xi-\eta|^2-\alpha|\eta|^2,
 		\qquad \alpha=\pm1.
 	\end{equation*}
 	
 	After removing the free evolution, the profile satisfies
 	\begin{equation*}	\partial_s f=-i\sum_{\alpha=\pm1}\lambda_\alpha
 		\e^{-is\Delta}\big(\e^{i\alpha s\Delta}f_\alpha\big)^2,
 	\end{equation*}
 	and hence the Duhamel formula takes the form
 	\begin{equation*}	\widehat f(t,\xi)=\widehat f(1,\xi)
 		-\frac{i}{2\pi}\sum_{\alpha=\pm1}\lambda_\alpha
 		\int_1^t\int_{\mathbb R^2}
 		e^{is\Phi_\alpha(\xi,\eta)}
 		\widehat f_\alpha(s,\xi-\eta)\widehat f_\alpha(s,\eta)
 		\,d\eta\,ds.
 	\end{equation*}
 	
 	Following the classical normal-form method of
 	Shatah \cite{Shatah1985}, we integrate by parts in time where
 	the phase does not vanish. Before specifying the frequency
 	cutoffs, we record the corresponding identity. If
 	$a_\alpha(s,\xi,\eta)$ is smooth and supported away from
 	$\{\Phi_\alpha=0\}$, then
 	\begin{equation}\label{eq:normal-general}
 		\begin{aligned}		&\int_1^t e^{is\Phi_\alpha}a_\alpha
 			\widehat f_\alpha(s,\xi-\eta)\widehat f_\alpha(s,\eta)\,ds\\
 			&=\left[e^{is\Phi_\alpha}\frac{a_\alpha}{i\Phi_\alpha}
 			\widehat f_\alpha(s,\xi-\eta)\widehat f_\alpha(s,\eta)
 			\right]_{s=1}^{s=t}\\
 			&\quad-\int_1^t e^{is\Phi_\alpha}
 			\bigg[\frac{\partial_sa_\alpha}{i\Phi_\alpha}
 			\widehat f_\alpha(s,\xi-\eta)\widehat f_\alpha(s,\eta) +\frac{a_\alpha}{i\Phi_\alpha}
 			\partial_s\big(\widehat f_\alpha(s,\xi-\eta)
 			\widehat f_\alpha(s,\eta)\big)\bigg] \,ds.
 		\end{aligned}
 	\end{equation}
 	Here all symbols are evaluated at $(s,\xi,\eta)$. We next identify where the phase can be inverted and where space nonresonance must be used instead.
 	
 	\subsection{Quadratic resonances}
 	We now examine the resonance geometry of the two quadratic phases. Direct calculation gives
 	\begin{align*}	\Phi_+(\xi,\eta)
 		&=2(\xi-\eta)\cdot\eta
 		=\big(|\xi-\eta|^2+|\eta|^2-|\xi-2\eta|^2\big),
 		\\
 		\Phi_-(\xi,\eta)
 		&=\big(|\xi|^2+|\xi-\eta|^2+|\eta|^2\big)
 		\ge\big(|\xi-\eta|^2+|\eta|^2\big),
 		\\
 		\nabla_\eta\Phi_\alpha(\xi,\eta)&=2\alpha(\xi-2\eta).
 	\end{align*}
 	Thus the space-resonant sets are $\{\xi=2\eta\}$ for both signs.
 	The plus phase vanishes on
 	$\{(\xi-\eta)\cdot\eta=0\}$, whereas the minus phase vanishes only
 	at $(\xi,\eta)=(0,0)$. In either case, the joint space--time resonant
 	set is the single point $(0,0)$. The minus interaction can therefore
 	be treated by a time normal form away from joint low frequencies,
 	while for the plus interaction one must also exploit space
 	nonresonance.
 	
 	This geometry motivates a decomposition into a neighborhood of the joint origin, a time-nonresonant region, and, for the plus interaction, a complementary space-nonresonant region.
 	
 	\subsection{Low-frequency and nonresonant regions}\label{subsec:partition}
 	Following the space--time resonance framework of
 	Germain--Masmoudi--Shatah \cite{GMS2D}, we introduce
 	cutoffs adapted to the resonance geometry. Choose a smooth
 	$0\le\chi\le1$ on $[0,\infty)$, equal to $1$ on $[0,1]$
 	and vanishing on $[4,\infty)$. Define
 	\begin{equation*}	\chi_{\mathrm{lo}}(s,\xi,\eta)
 		=\chi\big(s(|\xi-\eta|^2+|\eta|^2)\big).
 	\end{equation*}
 	Choose a smooth cutoff $0\le\vartheta\le1$ on $[0,\infty)$,
 	equal to $1$ on $[0,1/16]$ and vanishing on $[1/8,\infty)$.
 	We then define the time- and space-nonresonant cutoffs by
 	\begin{equation}\label{eq:partition}
 		\begin{aligned}		\chi_{T,+}&=(1-\chi_{\mathrm{lo}})
 			\vartheta\!\left(\frac{|\xi-2\eta|^2}{|\xi-\eta|^2+|\eta|^2}\right),\\
 			\chi_S&=(1-\chi_{\mathrm{lo}})
 			\left[1-\vartheta\!\left(
 			\frac{|\xi-2\eta|^2}{|\xi-\eta|^2+|\eta|^2}\right)\right],\\
 			\chi_{T,-}&=1-\chi_{\mathrm{lo}}.
 		\end{aligned}
 	\end{equation}
 	The products extend smoothly across the origin, because
 	$1-\chi_{\mathrm{lo}}$ vanishes near it. They satisfy
 	\begin{equation*}	1=\chi_{\mathrm{lo}}+\chi_{T,+}+\chi_S,
 		\qquad1=\chi_{\mathrm{lo}}+\chi_{T,-}.
 	\end{equation*}
 	The corresponding support conditions provide the quantitative bounds
 	\begin{align}	\Phi_+(\xi,\eta)
 		&\ge\frac78(|\xi-\eta|^2+|\eta|^2)
 		&&\text{on }\operatorname{supp}\chi_{T,+},\notag\\
 		\Phi_-(\xi,\eta)
 		&\ge(|\xi-\eta|^2+|\eta|^2)
 		&&\text{for all }\xi,\eta,\label{eq:timegapminus}\\
 		|\xi-2\eta|&\ge\frac14\sqrt{|\xi-\eta|^2+|\eta|^2}
 		&&\text{on }\operatorname{supp}\chi_S.\label{eq:spacegap}
 	\end{align}
 	These bounds justify integration by parts in time on the supports of $\chi_{T,\alpha}$ and in frequency on the support of $\chi_S$. We now apply the partition to the profile equation.
 	
 	\subsection{The working normal-form identity}
 	The preceding partition separates the quadratic terms according to the available oscillatory mechanism. Substituting \eqref{eq:partition} into the profile Duhamel formula and applying \eqref{eq:normal-general} to the time-nonresonant terms, we obtain
 	\begin{equation}\label{eq:fullNF}
 		\begin{aligned}		f(t)={}&f(1)
 			-i\sum_{\alpha=\pm1}\lambda_\alpha
 			\int_1^t \e^{-is\Delta}T_{\chi_{\mathrm{lo}}}
 			\big(\e^{i\alpha s\Delta}f_\alpha,\e^{i\alpha s\Delta}f_\alpha\big)\,ds\\
 			&-i\lambda_1\int_1^t \e^{-is\Delta}T_{\chi_S}
 			\big(\e^{is\Delta}f,\e^{is\Delta}f\big)\,ds\\
 			&-i\sum_{\alpha=\pm1}\lambda_\alpha
 			\left[\e^{-is\Delta}T_{\chi_{T,\alpha}/(i\Phi_\alpha)}
 			\big(\e^{i\alpha s\Delta}f_\alpha,\e^{i\alpha s\Delta}f_\alpha\big)
 			\right]_{s=1}^{s=t}\\
 			&+i\sum_{\alpha=\pm1}\lambda_\alpha\int_1^t
 			\e^{-is\Delta}T_{\partial_s\chi_{T,\alpha}/(i\Phi_\alpha)}
 			\big(\e^{i\alpha s\Delta}f_\alpha,\e^{i\alpha s\Delta}f_\alpha\big)\,ds\\
 			&+i\sum_{\alpha=\pm1}\lambda_\alpha\int_1^t
 			\e^{-is\Delta}T_{\chi_{T,\alpha}/(i\Phi_\alpha)}
 			\big(\e^{i\alpha s\Delta}\partial_sf_\alpha,\e^{i\alpha s\Delta}f_\alpha\big)\,ds\\
 			&+i\sum_{\alpha=\pm1}\lambda_\alpha\int_1^t
 			\e^{-is\Delta}T_{\chi_{T,\alpha}/(i\Phi_\alpha)}
 			\big(\e^{i\alpha s\Delta}f_\alpha,\e^{i\alpha s\Delta}\partial_sf_\alpha\big)\,ds.
 		\end{aligned}
 	\end{equation}
 	The first two integrals represent the low-frequency and
 	space-nonresonant quadratic contributions, respectively.
 	The third term is the normal-form boundary contribution
 	arising from integration by parts in time. The fourth contribution
 	is the cutoff-error integral, produced when the
 	time derivative falls on the explicitly time-dependent
 	symbol. Finally, the last two integrals arise when the
 	time derivative falls on either of the two profile
 	factors.
 	
 	To estimate \eqref{eq:fullNF}, we first establish uniform dyadic bounds for the normal-form symbols.
 	
 	\subsection{Dyadic normal-form symbols and their bounds}
 	To verify the multiplier estimates used later, we localize the time-normal-form divisor to dyadic joint-frequency scales. Choose a nonnegative $\varphi\in C_c^\infty((1/2,2))$ satisfying
 	$\sum_{N\in2^{\mathbb Z}}\varphi(r/N)=1$ for $r>0$. The symbol of
 	each dyadic normal-form piece is
 	\begin{equation}\label{eq:dyadicm}
 		m_{\alpha,s,N}(\xi,\eta)
 		=\frac{\chi_{T,\alpha}(s,\xi,\eta)}{i\Phi_\alpha(\xi,\eta)}
 		\varphi\!\left(\frac{\sqrt{|\xi-\eta|^2+|\eta|^2}}{N}\right).
 	\end{equation}
 	
 	The following lemma gives uniform bounds for the
 	dyadic normal-form symbols and their explicit
 	time derivatives.
 	
 	\begin{lemma}[Normal-form symbol bounds]\label{lem:normal-symbols}
 		Let $s\ge1$, $N\in2^{\mathbb Z}$, and
 		$\alpha\in\{+1,-1\}$. Suppose that
 		$m_{\alpha,s,N}$, defined in \eqref{eq:dyadicm},
 		is not identically zero. Then
 		\begin{equation}\label{eq:NF-frequency-range}
 			N\ge\frac1{2\sqrt s},\qquad sN^2\ge\frac14,
 		\end{equation}
 		and, for every multiindex $\beta$,
 		\begin{align}
 			|\partial_{\xi,\eta}^{\beta}m_{\alpha,s,N}|
 			&\lesssim_\beta N^{-2-|\beta|},
 			\label{eq:normal-symbol}\\
 			\operatorname{supp}\partial_sm_{\alpha,s,N}
 			&\subset\{\sqrt{|\xi-\eta|^2+|\eta|^2}
 			\approx N\approx s^{-1/2}\},
 			\notag\\[-0.3em]
 			|\partial_{\xi,\eta}^{\beta}\partial_sm_{\alpha,s,N}|
 			&\lesssim_\beta s^{-1}N^{-2-|\beta|}.
 			\label{eq:normal-symbol-time}
 		\end{align}
 	\end{lemma}
 	
 	\begin{proof}
 		A nonzero piece requires $1-\chi_{\mathrm{lo}}\ne0$, hence the
 		joint frequency exceeds $s^{-1/2}$. Its dyadic support lies below
 		$2N$, proving \eqref{eq:NF-frequency-range}. In the plus case,
 		the angular factor in $\chi_{T,+}$ divided by $\Phi_+$ is smooth on the unit joint-frequency sphere after
 		extension by zero near $\{\Phi_+=0\}$, and is homogeneous of degree
 		$-2$. In the minus case the same property holds for $1/\Phi_-$ by
 		\eqref{eq:timegapminus}. Thus differentiation of these homogeneous
 		symbols costs one inverse frequency per derivative.
 		
 		It remains to estimate the derivatives of the cutoff. Every positive-order frequency
 		derivative of $\chi(s(|\xi-\eta|^2+|\eta|^2))$ is supported where
 		the joint frequency is comparable to $s^{-1/2}$. On the support of a
 		nonzero dyadic block, this is $N\approx s^{-1/2}$, and scaling the
 		frequencies by $N$ gives uniform smooth symbol estimates.
 		Finally,
 		\begin{equation*}		\partial_s\chi_{\mathrm{lo}}
 			=(|\xi-\eta|^2+|\eta|^2)
 			\chi'\big(s(|\xi-\eta|^2+|\eta|^2)\big).
 		\end{equation*}
 		The time derivative is supported in
 		$1\le s(|\xi-\eta|^2+|\eta|^2)\le4$ and has size $O(s^{-1})$.
 		The same scaled-frequency calculation applies after further
 		frequency derivatives. Leibniz's rule then gives
 		\eqref{eq:normal-symbol}--\eqref{eq:normal-symbol-time}.
 	\end{proof}

 	\section{A priori estimates and proof of the main theorem}
 	\label{sec:reduction}
 	
 	In this section, we prove Theorem~\ref{thm:main}
 	assuming several analytic estimates established later
 	in the paper. We first collect these estimates,
 	then derive an a priori bound for the free profile,
 	and finally prove the main theorem by a bootstrap argument.
 	
 	For a solution on $[1,T]$, set
 	\begin{equation*}
 		A(t)=\|f(t)\|_{\mathcal Z},\qquad
 		X_T=\sup_{1\le s\le T}A(s).
 	\end{equation*}
 	Since the spatial cutoffs $\rho_j$ are real,
 	$\|f_+(s)\|_{\mathcal Z}=\|f_-(s)\|_{\mathcal Z}=A(s)$.
 	
 	\subsection{The analytic lemmas}
 	
 	We begin with the basic properties of the profile
 	space and the control of the initial profile
 	by the weighted Sobolev norm.
 	
 	\begin{lemma}[Profile norm and initial data]
 		\label{lem:basic}
 		The space $\mathcal Z$ is a separable Banach space,
 		$C_c^\infty$ is dense in it, and \eqref{eq:basic-intro} holds.
 		Moreover,
 		\begin{equation}\label{eq:initial-lemma}
 			\|\e^{-i\Delta}u_1\|_{\mathcal Z}\le C_0\|u_1\|_{\Sigma^2}
 		\end{equation}
 		for an absolute constant $C_0\ge1$.
 	\end{lemma}

 	\begin{proof}
 		The embeddings in \eqref{eq:basic-intro} were proved in
 		Subsection~\ref{subsec:profile-space}. Completeness follows
 		from the summable annular norm, while truncation and
 		smoothing give density and separability. It remains to
 		estimate the initial profile.
 		
 		For $g\in C_c^\infty(\mathbb R^2)$,
 		integration by parts gives
 		\begin{equation*}
 			\begin{aligned}
 				\norm{|x|\nabla g}_2^2
 				&=-\Rea\int |x|^2\bar g\,\Delta g\dd x+2\norm g_2^2\\
 				&\le\norm{|x|^2g}_2\norm{\Delta g}_2+2\norm g_2^2.
 			\end{aligned}
 		\end{equation*}
 		By approximation, this estimate extends to $\Sigma^2$.
 		The commutator identity
 		$x_\nu \e^{it\Delta}=\e^{it\Delta}(x_\nu-2it\partial_\nu)$
 		yields
 		\begin{equation}\label{eq:group-weights}
 			|x|^2\e^{-it\Delta}g
 			=\e^{-it\Delta}\big(|x|^2g+4itx\cdot\nabla g+4it g-4t^2\Delta g\big).
 		\end{equation}
 		Consequently,
 		\begin{equation}\label{eq:group-Sigma}
 			\norm{\e^{it\Delta}g}_{\Sigma^2}\lesssim(1+|t|)^2\norm g_{\Sigma^2}.
 		\end{equation}
 		In particular, taking $t=1$ in
 		\eqref{eq:group-weights} and using \eqref{eq:Zfrom2},
 		we obtain
 		\[
 		\norm{\e^{-i\Delta}u_1}_{\Z}
 		\lesssim\norm{\e^{-i\Delta}u_1}_2+\norm{|x|^2\e^{-i\Delta}u_1}_2
 		\lesssim\norm{u_1}_{\Sigma^2}.
 		\]
 		This proves \eqref{eq:initial-lemma}.
 	\end{proof}

 	We next state the estimate for the low-frequency
 	and space-nonresonant quadratic terms in the
 	normal-form identity \eqref{eq:fullNF}.
 	
 	\begin{lemma}[Low frequencies and space nonresonance]
 		\label{lem:space}
 		For $s\ge1$ and $g,h\in\mathcal Z$,
 		\begin{equation}\label{eq:spacelemma}
 			\begin{aligned}
 				&\|\e^{-is\Delta}T_{\chi_S(s)}(\e^{is\Delta}g,\e^{is\Delta}h)\|_{\mathcal Z}\\
 				&\quad+\sum_{\alpha=\pm1}
 				\|\e^{-is\Delta}T_{\chi_{\mathrm{lo}}(s)}
 				(\e^{i\alpha s\Delta}g,\e^{i\alpha s\Delta}h)\|_{\mathcal Z}
 				\lesssim s^{-1}\|g\|_{\mathcal Z}\|h\|_{\mathcal Z}.
 			\end{aligned}
 		\end{equation}
 		The constant is uniform in $s$ and no global second spatial moment is
 		assumed for either input.
 	\end{lemma}
 	
 	The proof is given in Section~\ref{sec:spaceproof},
 	where spatial localization, bilinear multiplier
 	estimates, and nonstationary phase arguments
 	are combined to obtain the required decay.
 	
 	We also need estimates for the boundary and
 	remainder terms arising from integration by parts
 	in time.  
 	
 	\begin{lemma}[Time normal-form estimates]
 		\label{lem:normal}
 		For either sign $\alpha$, $s\ge1$, and $g,h\in\mathcal Z$,
 		\begin{align}
 			&\|\e^{-is\Delta}T_{\chi_{T,\alpha}(s)/(i\Phi_\alpha)}
 			(\e^{i\alpha s\Delta}g,\e^{i\alpha s\Delta}h)\|_{\mathcal Z}
 			\lesssim (1+\log s)\|g\|_{\mathcal Z}\|h\|_{\mathcal Z},
 			\label{eq:normal-bound}\\
 			&\|\e^{-is\Delta}T_{\partial_s\chi_{T,\alpha}(s)/(i\Phi_\alpha)}
 			(\e^{i\alpha s\Delta}g,\e^{i\alpha s\Delta}h)\|_{\mathcal Z}
 			\lesssim s^{-1}\|g\|_{\mathcal Z}\|h\|_{\mathcal Z}.
 			\label{eq:cutoff-bound}
 		\end{align}
 		More generally, for $g\in\mathcal Z$ and $v\in L^\infty$,
 		\begin{equation}\label{eq:onesided-normal}
 			\|\e^{-is\Delta}T_{\chi_{T,\alpha}(s)/(i\Phi_\alpha)}
 			(\e^{i\alpha s\Delta}g,v)\|_{\mathcal Z}
 			\lesssim s(1+\log s)\|g\|_{\mathcal Z}\|v\|_\infty.
 		\end{equation}
 		The analogous estimate holds after exchanging the two inputs and
 		interchanging the corresponding variables of the multiplier.
 	\end{lemma}
 	
 	This lemma is proved in Subsection~\ref{sec:NFbounds},
 	using the normal-form symbol bounds from
 	Lemma~\ref{lem:normal-symbols} and the one-sided
 	frequency-localized estimate established in
 	Subsection~\ref{subsec:onesided}.
 	
 	Finally, we recall the standard local existence
 	and continuation properties needed to complete
 	the proof of the main theorem.
 	
 	\begin{lemma}[Local theory and continuation]
 		\label{lem:local}
 		For $u_1\in\Sigma^2$, there exists a unique maximal forward solution
 		$u\in C([1,T_{\max});\Sigma^2)\cap C^1([1,T_{\max});L^2)$,
 		and its profile belongs to $C^1([1,T];\Sigma^2)$ on each compact
 		subinterval of existence. If $T_{\max}<\infty$ while
 		$\sup_{1\le t<T_{\max}}\|u(t)\|_{H^2}<\infty$, the solution extends
 		past $T_{\max}$. Furthermore,
 		\begin{equation}\label{eq:energy-lemma}
 			\|u(t)\|_{H^2}\le\|u_1\|_{H^2}
 			\exp\!\left(C(|\lambda_1|+|\lambda_2|)
 			\int_1^t\|u(s)\|_\infty\,ds\right).
 		\end{equation}
 	\end{lemma}
 	
 	Local well-posedness and the continuation criterion in Lemma \ref{lem:local}
 	follow from the standard semilinear Schr\"odinger
 	theory in $H^2$ (see \cite[Sections 2.5 and 4.8]{Cazenave}),
 	together with persistence of polynomial weights
 	on finite time intervals. The energy estimate
 	\eqref{eq:energy-lemma} follows from the standard
 	quadratic product estimate in $H^2$ and
 	Gronwall's inequality.

 	\subsection{The a priori inequality}
 	
 	Following the almost-global bootstrap strategy of
 	Murphy--Pusateri \cite{MP}, we first derive an
 	a priori bound for the free profile; the continuity argument
 	is carried out in the final subsection.  
 	
 	\begin{proposition}\label{prop:apriori}
 		For any solution $u\in C([1,T];\Sigma^2)$ of \eqref{eq:pde}, we have
 		\begin{equation}\label{eq:apriori}
 			\begin{aligned}
 				A(t)\le{}&A(1)+C(|\lambda_1|+|\lambda_2|)
 				\big[A(1)^2+(1+\log t)A(t)^2\big]\\
 				&+C(|\lambda_1|+|\lambda_2|)
 				\int_1^t\frac{A(s)^2}{s}\,ds\\
 				&+C(|\lambda_1|+|\lambda_2|)^2
 				\int_1^t\frac{1+\log s}{s}A(s)^3\,ds,
 				\qquad 1\le t\le T.
 			\end{aligned}
 		\end{equation}
 		Consequently,
 		\begin{equation}\label{eq:apriori-sup}
 			\begin{aligned}
 				X_T\le A(1)&+C(|\lambda_1|+|\lambda_2|)
 				(1+\log T)X_T^2 +C(|\lambda_1|+|\lambda_2|)^2
 				(1+\log T)^2X_T^3.
 			\end{aligned}
 		\end{equation}
 		The constant $C$ is independent of $T$ and the solution.
 	\end{proposition}
 	
 	\begin{proof}
 		We first estimate the differentiated profile.
 		By the profile equation,
 		\[
 		\e^{is\Delta}\partial_s f
 		=-i(\lambda_1u^2+\lambda_2\overline u^{\,2}),
 		\]
 		and the dispersive decay estimate, we have
 		\begin{equation*}
 			\|\e^{i\alpha s\Delta}\partial_sf_\alpha\|_\infty
 			\lesssim (|\lambda_1|+|\lambda_2|)s^{-2}A(s)^2.
 		\end{equation*}
 		Applying \eqref{eq:onesided-normal}
 		and its exchanged-input version, we obtain
 		\begin{equation}\label{eq:cubic}
 			\begin{aligned}
 				&\|\e^{-is\Delta}T_{\chi_{T,\alpha}(s)/(i\Phi_\alpha)}
 				(\e^{i\alpha s\Delta}\partial_sf_\alpha,\e^{i\alpha s\Delta}f_\alpha)\|_{\mathcal Z}\\
 				&\quad+\|\e^{-is\Delta}T_{\chi_{T,\alpha}(s)/(i\Phi_\alpha)}
 				(\e^{i\alpha s\Delta}f_\alpha,\e^{i\alpha s\Delta}\partial_sf_\alpha)\|_{\mathcal Z}\\
 				&\qquad\lesssim (|\lambda_1|+|\lambda_2|)
 				\frac{1+\log s}{s}A(s)^3.
 			\end{aligned}
 		\end{equation}
 		
 		We now estimate the terms in the normal-form
 		identity \eqref{eq:fullNF}.
 		
 		By Lemma~\ref{lem:space}, the first two integrals,
 		corresponding to the low-frequency and
 		space-nonresonant contributions, are bounded by
 		\[
 		C(|\lambda_1|+|\lambda_2|)
 		\int_1^t\frac{A(s)^2}{s}\,ds.
 		\]
 		
 		The third term consists of the normal-form
 		boundary contributions. By
 		\eqref{eq:normal-bound}, their total size is at most
 		\[
 		C(|\lambda_1|+|\lambda_2|)
 		\big[(1+\log t)A(t)^2+A(1)^2\big].
 		\]
 		
 		The fourth term arises from the explicit time
 		derivative of the cutoff. Using
 		\eqref{eq:cutoff-bound}, its contribution is bounded by
 		\[
 		C(|\lambda_1|+|\lambda_2|)
 		\int_1^t\frac{A(s)^2}{s}\,ds.
 		\]
 		
 		Finally, the last two integrals are estimated
 		by \eqref{eq:cubic}. Including their outer
 		coefficients, their total contribution is bounded by
 		\[
 		C(|\lambda_1|+|\lambda_2|)^2
 		\int_1^t\frac{1+\log s}{s}A(s)^3\,ds.
 		\]
 		
 		Combining these estimates proves \eqref{eq:apriori}.
 		Taking the supremum over $1\le t\le T$
 		then yields \eqref{eq:apriori-sup}.
 	\end{proof}
 	
 	\subsection{Proof of Theorem~\ref{thm:main}}
 	
 	In this final subsection, we prove Theorem~\ref{thm:main}
 	by combining Proposition~\ref{prop:apriori}
 	with a standard bootstrap argument.
 	
 	\begin{proof}
 		By Lemma~\ref{lem:local}, a local solution exists
 		and $A(t)$ is continuous. Lemma~\ref{lem:basic}
 		also gives $A(1)\le C_0\varepsilon$.
 		
 		Absorbing the fixed coefficient sizes into
 		a constant $C_\lambda$, assume on an interval
 		of existence that $X_T\le4C_0\varepsilon$.
 		Then \eqref{eq:apriori-sup} implies
 		\begin{equation*}
 			X_T\le C_0\varepsilon
 			+C_\lambda(4C_0)^2\varepsilon^2(1+\log T)
 			+C_\lambda(4C_0)^3\varepsilon^3(1+\log T)^2.
 		\end{equation*}
 		For $T\le\exp(c/\varepsilon)$, we have
 		$\varepsilon(1+\log T)\le\varepsilon+c$.
 		Choose $c>0$ and then $\varepsilon_0>0$
 		sufficiently small so that
 		\[
 		C_\lambda(4C_0)^2(\varepsilon_0+c)
 		+C_\lambda(4C_0)^3(\varepsilon_0+c)^2\le C_0.
 		\]
 		Consequently, the bootstrap bound improves to
 		\begin{equation}\label{eq:bootstrap-improved}
 			X_T\le2C_0\varepsilon.
 		\end{equation}
 		
 		A continuity argument gives the improved bound on
 		every compact subinterval of the lifespan that lies
 		within the target time interval.
 		Combining the resulting dispersive decay with
 		\eqref{eq:energy-lemma}, we obtain
 		\begin{equation}\label{eq:H2-final}
 			\|u(t)\|_{H^2}\le\varepsilon
 			\exp(C_\lambda\varepsilon\log t)
 			\le C_\lambda\varepsilon
 		 \quad\text{for all}\quad 1\le t<\min\{T_{\max},e^{c/\varepsilon}\}.
 		\end{equation}
 		The continuation criterion in
 		Lemma~\ref{lem:local} therefore excludes
 		finite-time breakdown at or before $e^{c/\varepsilon}$.
 		Finally, \eqref{eq:basic-intro},
 		\eqref{eq:bootstrap-improved}, and
 		\eqref{eq:H2-final} give \eqref{eq:main}.
 		When both coefficients vanish, the conclusion
 		follows directly from the free evolution.
 	\end{proof}

 	It remains to prove the estimates stated in
 	Lemmas~\ref{lem:space} and~\ref{lem:normal}.
 	We prove Lemma~\ref{lem:normal} in
 	Subsection~\ref{sec:NFbounds}, using the one-sided
 	weighted multiplier estimate established in
 	Subsection~\ref{subsec:onesided}, and prove
 	Lemma~\ref{lem:space} in Section~\ref{sec:spaceproof}.

 	\section{One-sided multiplier estimates and bounds for the time normal form}
 	\label{sec:weightedNF}

 	We first establish a one-sided weighted multiplier
 	estimate and then apply it to the time normal form.
 	\subsection{The one-sided weighted estimate}\label{subsec:onesided}
 	Following the Fourier multiplier approach in
 	\cite{GMS2D,MP}, we combine bilinear multiplier
 	estimates with spatial annular localization.
 	The resulting estimate requires the $\mathcal Z$
 	norm on only one input, while the other is
 	controlled in $L^\infty$.
 	We begin with a single frequency block.

 	\begin{lemma}[One-sided estimate on a frequency block]
 		\label{lem:onesided}
 		Let $s\ge1$, $N>0$, and suppose that $sN^2\ge c_*>0$.
 		Let $m_N(\xi,\eta)$ be a smooth symbol supported in
 		\[
 		|\xi-\eta|+|\eta|\le C_*N
 		\]
 		and satisfying
 		\begin{equation}\label{eq:onesided-symbol}
 			|\partial_{\xi,\eta}^{\beta}m_N(\xi,\eta)|
 			\le C_\beta N^{-2-|\beta|}
 		\end{equation}
 		for all multiindices $\beta$ up to a sufficiently large
 		fixed order. Then, for $\alpha=\pm1$,
 		\begin{equation}\label{eq:onesided}
 			\|\e^{-is\Delta}T_{m_N}(\e^{i\alpha s\Delta}g,v)\|_{\mathcal Z}
 			\lesssim
 			s\min(1,N^{-1})\|g\|_{\mathcal Z}\|v\|_\infty.
 		\end{equation}
 		The implicit constant may depend on $c_*$, $C_*$,
 		and the symbol bounds, but is independent of
 		$s,N,g$, and $v$.
 	\end{lemma}
 	
 	\begin{proof}
 		We prove the estimate by first localizing $g$
 		to a spatial annulus. For a fixed spatial atom
 		$g_j=\rho_jg$, set
 		\begin{equation*}		R=\max(2^j,sN),\qquad
 			w=\e^{-is\Delta}T_{m_N}(\e^{i\alpha s\Delta}g_j,v).
 		\end{equation*}
 		The scale $R$ accounts for both the spatial
 		localization of $g_j$ and the oscillations
 		generated by the Schr\"odinger evolution.
 		
 		We begin with the $L^2$ estimate.
 		Since the support of $m_N$ restricts the first input
 		frequency to $O(N)$, we may insert a smooth projection
 		$P_{\le CN}$ into the first input. Applying
 		Lemma~\ref{lem:compact}, the unitarity of $\e^{it\Delta}$
 		on $L^2$, and \eqref{eq:atom-low}, we obtain
 		\begin{align}		\|w\|_2
 			&\lesssim
 			N^{-2}\|P_{\le CN}\e^{i\alpha s\Delta}g_j\|_2
 			\|v\|_\infty \notag\\
 			&\lesssim
 			N^{-2}\min(2^{-j},N)\,
 			2^j\|g_j\|_2\|v\|_\infty.
 			\label{eq:one-L2}
 		\end{align}
 		
 		To recover the $\mathcal Z$ norm through
 		Lemma~\ref{lem:moment}, we also need an estimate
 		for the second spatial moment of $w$.
 		By Plancherel's theorem, it suffices to control
 		the second derivatives of $\widehat w$ with
 		respect to the output frequency.

 		We may assume that $v$ is smooth, since the general
 		case follows by a standard regularization argument.
 		The Fourier representation of $w$ is
 		\begin{equation}\label{eq:one-Fourier}
 			\widehat w(\xi)
 			=\frac1{2\pi}\int_{\mathbb R^2}
 			e^{is(|\xi|^2-\alpha|\xi-\eta|^2)}
 			m_N(\xi,\eta)
 			\widehat g_j(\xi-\eta)\widehat v(\eta)\,d\eta.
 		\end{equation}
 		
 		To organize the output derivatives, for
 		$\nu=1,2$ introduce the differentiated symbols
 		\begin{align*}		b_{\nu,N}
 			&=\partial_{\xi_\nu}m_N
 			+2is\big[\xi_\nu-\alpha(\xi_\nu-\eta_\nu)\big]m_N,
 			\\
 			c_{\nu,N}
 			&=\partial_{\xi_\nu}b_{\nu,N}
 			+2is\big[\xi_\nu-\alpha(\xi_\nu-\eta_\nu)\big]b_{\nu,N}.
 		\end{align*}
 		Differentiating \eqref{eq:one-Fourier} twice with respect
 		to $\xi_\nu$, while keeping $\eta$ fixed, gives
 		\begin{equation}\label{eq:one-second}
 			\begin{aligned}			\partial_{\xi_\nu}^2\widehat w(\xi)
 				=\frac1{2\pi}\int_{\mathbb R^2}
 				&e^{is(|\xi|^2-\alpha|\xi-\eta|^2)}
 				\Big[
 				m_N\partial_{\xi_\nu}^2\widehat g_j(\xi-\eta)\\
 				&\quad
 				+2b_{\nu,N}\partial_{\xi_\nu}
 				\widehat g_j(\xi-\eta)
 				+c_{\nu,N}\widehat g_j(\xi-\eta)
 				\Big]\widehat v(\eta)\,d\eta.
 			\end{aligned}
 		\end{equation}

 		We next verify that the differentiated symbols
 		satisfy the bounds required by Lemma~\ref{lem:compact}.
 		Since $sN^2\ge c_*$, we have
 		\[
 		N^{-1}\lesssim sN\le R,
 		\qquad
 		s\lesssim(sN)^2\le R^2,
 		\]
 		where the implicit constants may depend on $c_*$.
 		Consequently, \eqref{eq:onesided-symbol} and
 		Leibniz's rule yield
 		\begin{align*}		|\partial_{\xi,\eta}^{\beta}b_{\nu,N}|
 			&\lesssim_\beta
 			(N^{-1}+sN)N^{-2-|\beta|}
 			\lesssim_\beta RN^{-2-|\beta|},
 			\\
 			|\partial_{\xi,\eta}^{\beta}c_{\nu,N}|
 			&\lesssim_\beta
 			(N^{-2}+s+s^2N^2)N^{-2-|\beta|}
 			\lesssim_\beta R^2N^{-2-|\beta|}.
 		\end{align*}
 		
 		These bounds allow us to apply
 		Lemma~\ref{lem:compact} separately to the
 		three terms in \eqref{eq:one-second}.
 		As before, the support conditions permit
 		the insertion of a smooth projection
 		$P_{\le CN}$ into the first input.
 		By Plancherel's theorem,
 		\begin{align*}		\|\partial_{\xi_\nu}^2\widehat w\|_2
 			\lesssim N^{-2}\|v\|_\infty
 			\Big(
 			&\|P_{\le CN}(x_\nu^2g_j)\|_2
 			+R\|P_{\le CN}(x_\nu g_j)\|_2
 			+R^2\|P_{\le CN}g_j\|_2
 			\Big).
 		\end{align*}
 		Applying the atomic estimate \eqref{eq:atom-low},
 		we further obtain
 		\[
 		\begin{aligned}		\|\partial_{\xi_\nu}^2\widehat w\|_2
 			\lesssim{}&
 			N^{-2}\big(2^{2j}+R2^j+R^2\big)
 			2^j\min(2^{-j},N)
 			\|g_j\|_2\|v\|_\infty.
 		\end{aligned}
 		\]
 		Since $2^j\le R$, the expression in parentheses
 		is at most $3R^2$. Summing over $\nu=1,2$
 		and using Plancherel's theorem gives
 		\begin{equation}\label{eq:one-moment}
 			\||x|^2w\|_2
 			\lesssim
 			R^2N^{-2}\min(2^{-j},N)
 			2^j\|g_j\|_2\|v\|_\infty.
 		\end{equation}
 		
 		We now combine the $L^2$ and second-moment
 		estimates to recover the annular norm.
 		By \eqref{eq:one-L2}, \eqref{eq:one-moment},
 		and Lemma~\ref{lem:moment} with spatial scale $R$,
 		we obtain
 		\begin{equation}\label{eq:one-before-algebra}
 			\|w\|_{\mathcal Z}
 			\lesssim
 			RN^{-2}\min(2^{-j},N)
 			2^j\|g_j\|_2\|v\|_\infty.
 		\end{equation}
 		
 		We claim that
 		\begin{equation}\label{eq:one-algebra}
 			RN^{-2}\min(2^{-j},N)
 			\lesssim s\min(1,N^{-1}).
 		\end{equation}
 		Recall that $R=\max(2^j,sN)$.
 		We distinguish two cases.
 		
 		If $2^j\le sN$, then $R=sN$. Since
 		$j\ge0$, we have $2^{-j}\le1$, and hence
 		\[
 		\begin{aligned}		RN^{-2}\min(2^{-j},N)
 			&\le sN^{-1}\min(1,N)
 			=s\min(1,N^{-1}).
 		\end{aligned}
 		\]
 		
 		If instead $2^j>sN$, then $R=2^j$.
 		Using $sN^2\ge c_*$ and $s\ge1$, we obtain
 		\[
 		RN^{-2}\min(2^{-j},N)
 		\le N^{-2}
 		\lesssim s\min(1,N^{-1}).
 		\]
 		Thus \eqref{eq:one-algebra} holds in both cases.
 		
 		Substituting \eqref{eq:one-algebra} into
 		\eqref{eq:one-before-algebra}, we obtain
 		\[
 		\|\e^{-is\Delta}T_{m_N}
 		(\e^{i\alpha s\Delta}g_j,v)\|_{\mathcal Z}
 		\lesssim
 		s\min(1,N^{-1})2^j\|g_j\|_2\|v\|_\infty.
 		\]
 		Finally, summing over $j\ge0$ and recalling
 		\[
 		\|g\|_{\mathcal Z}
 		=\sum_{j\ge0}2^j\|g_j\|_2,
 		\]
 		we arrive at \eqref{eq:onesided}.
 	\end{proof}

 	\subsection{Bounds for the time normal form}\label{sec:NFbounds}
 	
 	We now prove Lemma~\ref{lem:normal}.
 	The main ingredient is the one-sided weighted multiplier
 	estimate from Subsection~\ref{subsec:onesided},
 	together with the symbol bounds in
 	Lemma~\ref{lem:normal-symbols}. 
 	
 	\begin{proof}[Proof of Lemma~\ref{lem:normal}]
 		We begin with the one-sided estimate
 		\eqref{eq:onesided-normal}.
 		By Lemma~\ref{lem:normal-symbols}, the dyadic
 		normal-form symbols $m_{\alpha,s,N}$ defined in
 		\eqref{eq:dyadicm} satisfy $sN^2\ge1/4$
 		and the hypotheses of \eqref{eq:onesided}.
 		Thus, for $g\in\mathcal Z$ and $v\in L^\infty$,
 		\[
 		\|\e^{-is\Delta}T_{m_{\alpha,s,N}}(\e^{i\alpha s\Delta}g,v)\|_{\mathcal Z}
 		\lesssim s\min(1,N^{-1})\|g\|_{\mathcal Z}\|v\|_\infty.
 		\]
 		
 		To sum these estimates, we use the lower bound
 		on the admissible frequency scales.
 		Since $N\in2^{\mathbb Z}$, the dyadic sum satisfies
 		\begin{equation*}		\begin{aligned}			\sum_{N\ge(2\sqrt s)^{-1}}\min(1,N^{-1})
 				&=\sum_{(2\sqrt s)^{-1}\le N\le1}1+\sum_{N>1}N^{-1}\\
 				&\lesssim1+\log s.
 			\end{aligned}
 		\end{equation*}
 		Summing the frequency-localized estimates
 		therefore proves \eqref{eq:onesided-normal}.
 		
 		We next establish the quadratic normal-form
 		bound \eqref{eq:normal-bound}.
 		Choosing $v=\e^{i\alpha s\Delta}h$ and using
 		$\|\e^{i\alpha s\Delta}h\|_\infty
 		\lesssim s^{-1}\|h\|_{\mathcal Z}$,
 		we obtain \eqref{eq:normal-bound}.
 		
 		It remains to control the term in which the
 		time derivative falls on the cutoff.
 		By \eqref{eq:normal-symbol-time},
 		the differentiated symbol gains an additional
 		factor $s^{-1}$ and is supported where
 		$N\approx s^{-1/2}$.
 		Hence only $O(1)$ dyadic frequency blocks
 		contribute. Applying the same one-sided block
 		estimate and the dispersive decay bound gives
 		\[
 		\|\e^{-is\Delta}T_{\partial_s\chi_{T,\alpha}(s)/(i\Phi_\alpha)}
 		(\e^{i\alpha s\Delta}g,\e^{i\alpha s\Delta}h)\|_{\mathcal Z}
 		\lesssim s^{-1}\|g\|_{\mathcal Z}\|h\|_{\mathcal Z}.
 		\]
 		This establishes \eqref{eq:cutoff-bound}
 		and completes the proof.
 	\end{proof}

 	\section{Low frequencies and the space-nonresonant term}\label{sec:spaceproof}
 	
 	In this section, we prove Lemma~\ref{lem:space} by estimating the
 	low-frequency and space-nonresonant contributions separately.
 	The argument is carried out at a fixed time $s\ge1$.
 	
 	We first reduce the problem to spatial atoms. Fix
 	$g_j=\rho_jg$ and $h_k=\rho_kh$. Since the multipliers under
 	consideration are symmetric in the two input frequencies, we may
 	assume $j\ge k$. For this pair of atoms, set
 	\begin{equation}\label{eq:space-RK}
 		R=\max\{2^j,2^k,\sqrt s\}=\max\{2^j,\sqrt s\},
 		\qquad K=\frac{128R}{s}.
 	\end{equation}
 	Here $R$ is the relevant spatial scale, while $K$ separates
 	moderate from high joint frequencies. We write $N\in2^{\mathbb Z}$
 	for a dyadic joint-frequency scale, corresponding to
 	\[
 	\frac N2<\sqrt{|\xi-\eta|^2+|\eta|^2}<2N.
 	\]
 	Thus ``high frequency'' means high relative to $K$, not
 	necessarily $N\ge1$. The definitions give
 	\begin{equation}\label{eq:space-basic-scales}
 		R\ge1,\qquad R\ge\sqrt s,\qquad
 		K\ge128s^{-1/2},\qquad sK=128R.
 	\end{equation}
 	
 	It suffices to establish, for each fixed pair of spatial atoms,
 	\begin{equation}\label{eq:space-atom-goal}
 		\|\e^{-is\Delta}T_{\chi_S}
 		(\e^{is\Delta}g_j,\e^{is\Delta}h_k)\|_{\mathcal Z}
 		\lesssim\frac{2^{j+k}}s\norm{g_j}_2\norm{h_k}_2,
 	\end{equation}
 	and, for each $\alpha=\pm1$,
 	\begin{equation}\label{eq:res-final}
 		\|\e^{-is\Delta}T_{\chi_{\mathrm{lo}}}
 		(\e^{i\alpha s\Delta}g_j,\e^{i\alpha s\Delta}h_k)\|_{\mathcal Z}
 		\lesssim\frac{2^{j+k}}s\norm{g_j}_2\norm{h_k}_2.
 	\end{equation}
 	The low-frequency estimate \eqref{eq:res-final} is proved in
 	Subsection~\ref{subsec:space-low}. To establish
 	\eqref{eq:space-atom-goal}, we treat moderate and high
 	frequencies in Subsections~\ref{subsec:space-moderate}
 	and \ref{subsec:space-high}, respectively. We then sum
 	the resulting estimates over the spatial atoms to conclude
 	Lemma~\ref{lem:space}. All implicit constants are independent
 	of $s,j,k$, and the dyadic scale $N$.

 	\subsection{The low-frequency contribution}\label{subsec:space-low}
 	
 	We begin with the low-frequency estimate \eqref{eq:res-final}.
 	Fix $\alpha\in\{+1,-1\}$ and set
 	\[
 	w=\e^{-is\Delta}T_{\chi_{\mathrm{lo}}}(\e^{i\alpha s\Delta}g_j,\e^{i\alpha s\Delta}h_k).
 	\]
 	Its Fourier transform is
 	\[
 	\widehat w(\xi)=\frac1{2\pi}\int_{\mathbb R^2}
 	\e^{is\Phi_\alpha(\xi,\eta)}
 	\chi_{\mathrm{lo}}(s,\xi,\eta)
 	\widehat g_j(\xi-\eta)\widehat h_k(\eta)\dd\eta.
 	\]
 	Since $\Phi_\alpha$ is homogeneous of degree two,
 	\begin{equation*}
 		\begin{aligned}
 			\e^{is\Phi_\alpha(\xi,\eta)}\chi_{\mathrm{lo}}(s,\xi,\eta)
 			={}&\e^{i\Phi_\alpha(\sqrt s\,\xi,\sqrt s\,\eta)}
 			\chi\big(|\sqrt s(\xi-\eta)|^2+|\sqrt s\,\eta|^2\big).
 		\end{aligned}
 	\end{equation*}
 	The right-hand side is a fixed smooth, compactly
 	supported function of the scaled variables.
 	Consequently, for every multiindex $\beta$,
 	\begin{equation*}
 		\left|\partial_\xi^\beta
 		\big(\e^{is\Phi_\alpha}\chi_{\mathrm{lo}}\big)\right|
 		\lesssim_\beta s^{|\beta|/2},
 		\quad\text{and}\quad |\xi-\eta|,|\eta|\lesssim s^{-1/2}
 		\quad\text{on its support}.
 	\end{equation*}
 	Differentiating the Fourier representation with
 	respect to $\xi$, while keeping $\eta$ fixed,
 	and applying Young's convolution inequality
 	$L^2*L^1\to L^2$, we obtain, for $|\beta|\le2$,
 	\begin{equation}\label{eq:res-young}
 		\begin{aligned}
 			\norm{\partial_\xi^\beta\widehat w}_2
 			\lesssim\sum_{\gamma\le\beta}
 			&s^{(|\beta|-|\gamma|)/2}
 			\norm{\one_{|\xi|\le2s^{-1/2}}
 				\partial_\xi^\gamma\widehat g_j}_2
 			\norm{\one_{|\xi|\le2s^{-1/2}}\widehat h_k}_1.
 		\end{aligned}
 	\end{equation}
 	
 	We first estimate the last factor using the area
 	of the two-dimensional frequency ball:
 	\begin{equation}\label{eq:res-second}
 		\norm{\one_{|\xi|\le2s^{-1/2}}\widehat h_k}_1
 		\lesssim s^{-1}\norm{\widehat h_k}_\infty
 		\lesssim s^{-1}2^k\norm{h_k}_2.
 	\end{equation}
 	For the first factor, Plancherel's theorem gives
 	\[
 	\norm{\one_{|\xi|\le2s^{-1/2}}
 		\partial_\xi^\gamma\widehat g_j}_2
 	\le\norm{x^\gamma g_j}_2
 	\lesssim2^{j|\gamma|}\norm{g_j}_2,
 	\]
 	while the Fourier $L^\infty$ estimate, together
 	with the square root of the ball's area, yields
 	\[
 	\norm{\one_{|\xi|\le2s^{-1/2}}
 		\partial_\xi^\gamma\widehat g_j}_2
 	\lesssim s^{-1/2}\norm{x^\gamma g_j}_1
 	\lesssim s^{-1/2}2^{j(|\gamma|+1)}\norm{g_j}_2.
 	\]
 	Taking the minimum of these two bounds, as in
 	\eqref{eq:atom-low}, we obtain
 	\begin{equation}\label{eq:res-first}
 		\norm{\one_{|\xi|\le2s^{-1/2}}
 			\partial_\xi^\gamma\widehat g_j}_2
 		\lesssim
 		2^{j(|\gamma|+1)}\min(2^{-j},s^{-1/2})\norm{g_j}_2.
 	\end{equation}
 	
 	Substituting \eqref{eq:res-second} and
 	\eqref{eq:res-first} into \eqref{eq:res-young},
 	we arrive at
 	\begin{equation}\label{eq:res-substitute}
 		\begin{aligned}
 			\norm{\partial_\xi^\beta\widehat w}_2
 			\lesssim{}&\frac{2^{j+k}}s
 			\min(2^{-j},s^{-1/2})\norm{g_j}_2\norm{h_k}_2
 			\sum_{\gamma\le\beta}
 			s^{(|\beta|-|\gamma|)/2}2^{j|\gamma|}.
 		\end{aligned}
 	\end{equation}
 	By the definition of $R$ in \eqref{eq:space-RK},
 	\begin{equation*}		\min(2^{-j},s^{-1/2})=R^{-1},\qquad
 		s^{(|\beta|-|\gamma|)/2}2^{j|\gamma|}
 		=(\sqrt s)^{|\beta|-|\gamma|}(2^j)^{|\gamma|}
 		\le R^{|\beta|}.
 	\end{equation*}
 	Since $|\beta|\le2$, the sum in
 	\eqref{eq:res-substitute} has only finitely many terms. It follows that
 	\begin{equation}\label{eq:res-moments}
 		\norm{\partial_\xi^\beta\widehat w}_2
 		\lesssim R^{|\beta|-1}\frac{2^{j+k}}s
 		\norm{g_j}_2\norm{h_k}_2,
 		\qquad |\beta|\le2.
 	\end{equation}
 	Finally, applying Lemma~\ref{lem:moment} with
 	$L=R$ to the zeroth-order and pure second-order
 	estimates in \eqref{eq:res-moments} proves
 	\eqref{eq:res-final}.
 	
 	\subsection{Moderate frequencies}\label{subsec:space-moderate}
 	In this subsection, we estimate the moderate-frequency part
 	of \eqref{eq:space-atom-goal} by applying the classical
 	Coifman--Meyer theorem in the form recalled in
 	Subsection~\ref{subsec:multiplier-bounds}.
 	Let $0\le\zeta\le1$ be a fixed smooth cutoff satisfying
 	$\zeta=1$ on $[0,1]$ and $\zeta=0$ on $[2,\infty)$.
 	We split the space-nonresonant symbol using the cutoffs
 	\[
 	\zeta\!\left(\frac{\sqrt{|\xi-\eta|^2+|\eta|^2}}{K}\right)
 	\quad\text{and}\quad
 	1-\zeta\!\left(\frac{\sqrt{|\xi-\eta|^2+|\eta|^2}}{K}\right).
 	\]
 	For the moderate-frequency contribution, set
 	\begin{equation*}
 		\begin{aligned}
 			m(\xi,\eta)
 			&=\chi_S(s,\xi,\eta)
 			\zeta\!\left(\frac{\sqrt{|\xi-\eta|^2+|\eta|^2}}{K}\right),\\
 			w&=\e^{-is\Delta}T_m(\e^{is\Delta}g_j,\e^{is\Delta}h_k).
 		\end{aligned}
 	\end{equation*}
 	For convenience, throughout this subsection write
 	\[
 	r=\sqrt{|\xi-\eta|^2+|\eta|^2}.
 	\]
 	By the support properties of the cutoffs, $m$ and all its
 	derivatives are supported in
 	\begin{equation}\label{eq:moderate-band}
 		s^{-1/2}\le r\le2K.
 	\end{equation}
 	
 	We first establish uniform derivative bounds for the symbol $m$.
 	Recall that $\vartheta$ is the smooth cutoff fixed in
 	Subsection~\ref{subsec:partition}, equal to $1$ on $[0,1/16]$
 	and vanishing on $[1/8,\infty)$.
 	The angular factor
 	$1-\vartheta(|\xi-2\eta|^2/(|\xi-\eta|^2+|\eta|^2))$
 	is smooth and homogeneous of degree zero away from the origin. Moreover, for
 	$|\beta|\ge1$, derivatives of the lower and upper radial
 	cutoffs are supported where $r\approx s^{-1/2}$ and
 	$r\approx K$, respectively. Consequently,
 	\[
 	\begin{aligned}		\left|\partial_{\xi,\eta}^\beta\left[1-\vartheta\!\left(\frac{|\xi-2\eta|^2}{|\xi-\eta|^2+|\eta|^2}\right)\right]\right|
 		&\lesssim_\beta r^{-|\beta|}, &&r>0,\\
 		|\partial_{\xi,\eta}^\beta[1-\chi(sr^2)]|
 		&\lesssim_\beta r^{-|\beta|}, &&r\ge s^{-1/2},\\
 		|\partial_{\xi,\eta}^\beta\zeta(r/K)|
 		&\lesssim_\beta r^{-|\beta|}, &&0<r\le2K.
 	\end{aligned}
 	\]
 	The corresponding zeroth-order bounds are immediate.
 	Applying Leibniz's rule, we therefore obtain
 	\begin{equation}\label{eq:moderate-symbol}
 		|\partial_{\xi,\eta}^\beta m(\xi,\eta)|
 		\lesssim_\beta r^{-|\beta|}
 		\qquad (r>0).
 	\end{equation}
 	In particular, on a joint-frequency shell $r\approx N$,
 	we have $|\partial_{\xi,\eta}^\beta m|
 	\lesssim_\beta N^{-|\beta|}$.
 	
 	Next, we estimate the symbols arising from differentiation
 	with respect to the output frequency. For $\nu=1,2$,
 	introduce the auxiliary symbols
 	\begin{equation*}		b_\nu=\partial_{\xi_\nu}m+2is\eta_\nu m,
 		\qquad
 		c_\nu=\partial_{\xi_\nu}b_\nu+2is\eta_\nu b_\nu.
 	\end{equation*}
 	Differentiating the Fourier representation twice with
 	respect to $\xi_\nu$, while keeping $\eta$ fixed, yields
 	\begin{equation}\label{eq:moderate-second}
 		\begin{aligned}
 			\partial_{\xi_\nu}^2\widehat w(\xi)
 			=\frac1{2\pi}\int_{\R^2}e^{2is(\xi-\eta)\cdot\eta}
 			\big[&m\,\partial_{\xi_\nu}^2\widehat g_j(\xi-\eta)
 			+2b_\nu\,\partial_{\xi_\nu}\widehat g_j(\xi-\eta)\\
 			&+c_\nu\,\widehat g_j(\xi-\eta)\big]
 			\widehat h_k(\eta)\dd\eta.
 		\end{aligned}
 	\end{equation}
 	
 	To apply the Coifman--Meyer estimate to these terms,
 	we need uniform bounds for the normalized symbols
 	$m$, $R^{-1}b_\nu$, and $R^{-2}c_\nu$.
 	By \eqref{eq:moderate-symbol}, the linear dependence
 	on $\eta_\nu$, and Leibniz's rule, we have, for every
 	further multiindex $\beta$,
 	\begin{equation}\label{eq:polynomial-symbol}
 		\begin{aligned}
 			|\partial_{\xi,\eta}^\beta b_\nu|
 			&\lesssim_\beta (r^{-1}+sr)r^{-|\beta|},\\
 			|\partial_{\xi,\eta}^\beta c_\nu|
 			&\lesssim_\beta (r^{-2}+s+s^2r^2)r^{-|\beta|}.
 		\end{aligned}
 	\end{equation}
 	Recall that $R=\max\{2^j,\sqrt s\}$ and
 	$K=128R/s$. On the support specified in
 	\eqref{eq:moderate-band}, these definitions imply
 	\begin{equation*}
 		\frac1{Rr}\le1,\qquad
 		\frac{sr}{R}\le256,\qquad
 		\frac{s}{R^2}\le1.
 	\end{equation*}
 	Combining these inequalities with
 	\eqref{eq:moderate-symbol} and
 	\eqref{eq:polynomial-symbol}, we obtain
 	\begin{equation*}		|\partial_{\xi,\eta}^\beta m|
 		+|\partial_{\xi,\eta}^\beta(R^{-1}b_\nu)|
 		+|\partial_{\xi,\eta}^\beta(R^{-2}c_\nu)|
 		\lesssim_\beta r^{-|\beta|}.
 	\end{equation*}
 	Thus the three normalized symbols satisfy uniform
 	Coifman--Meyer bounds, including all the required
 	higher-order derivative bounds. The corresponding
 	operator norms are therefore independent of $s,j$, and $k$.
 	
 	Finally, we estimate the $L^2$ norm and the second
 	spatial moment of $w$. Lemma~\ref{lem:moment} will then
 	yield the required $\mathcal Z$ bound.
 	The support of these symbols restricts the first input
 	frequency to $\lesssim K$. We may therefore insert a
 	smooth projection $P_{\le CK}$ into that input after
 	performing the output differentiations.
 	Applying \eqref{eq:CM}, \eqref{eq:atom-low},
 	and \eqref{eq:atom-disp}, we first obtain
 	\begin{equation}\label{eq:moderate-L2}
 		\begin{aligned}
 			\norm w_2
 			&\lesssim\norm{P_{\le CK}g_j}_2\norm{\e^{is\Delta}h_k}_\infty\\
 			&\lesssim\frac{2^{j+k}}s\min(2^{-j},K)
 			\norm{g_j}_2\norm{h_k}_2.
 		\end{aligned}
 	\end{equation}
 	For the second derivative in \eqref{eq:moderate-second},
 	we apply the same multiplier estimate to
 	$m$, $R^{-1}b_\nu$, and $R^{-2}c_\nu$.
 	Plancherel's theorem and the localized atomic bounds
 	then give
 	\begin{equation*}
 		\begin{aligned}
 			\norm{\partial_{\xi_\nu}^2\widehat w}_2
 			\lesssim{}&\bigl[
 			\norm{P_{\le CK}(x_\nu^2g_j)}_2
 			+R\norm{P_{\le CK}(x_\nu g_j)}_2
 			+R^2\norm{P_{\le CK}g_j}_2\bigr]
 			\norm{\e^{is\Delta}h_k}_\infty\\
 			\lesssim{}&\frac{2^{j+k}}s
 			(2^{2j}+R2^j+R^2)\min(2^{-j},K)
 			\norm{g_j}_2\norm{h_k}_2.
 		\end{aligned}
 	\end{equation*}
 	Since $2^j\le R$, the expression in parentheses
 	is bounded by $3R^2$. Summing over $\nu=1,2$
 	and using Plancherel's theorem, we obtain
 	\begin{equation*}
 		\norm{|x|^2w}_2
 		\lesssim R^2\frac{2^{j+k}}s\min(2^{-j},K)
 		\norm{g_j}_2\norm{h_k}_2.
 	\end{equation*}
 	Applying Lemma~\ref{lem:moment} with $L=R$
 	and using \eqref{eq:moderate-L2}, we conclude that
 	\[
 	\norm w_{\Z}\lesssim
 	\frac{R2^{j+k}}s\min(2^{-j},K)
 	\norm{g_j}_2\norm{h_k}_2.
 	\]
 	
 	It remains to control the factor
 	$R\min(2^{-j},K)$. If $R=2^j$, then
 	$R\min(2^{-j},K)\le1$. If $R=\sqrt s$,
 	then $K=128/R$, and hence
 	$R\min(2^{-j},K)\le128$. Thus this factor
 	is uniformly bounded, and we obtain
 	\begin{equation}\label{eq:moderate-final}
 		\norm w_{\Z}\lesssim\frac{2^{j+k}}s\norm{g_j}_2\norm{h_k}_2.
 	\end{equation}

 	\subsection{High frequencies}\label{subsec:space-high}
 	Following the space-nonresonant integration-by-parts
 	method of Germain--Masmoudi--Shatah \cite{GMS2D},
 	we prove the high-frequency part of
 	\eqref{eq:space-atom-goal} using the complete oscillatory
 	phase, which also incorporates the locations of the
 	spatial atoms. We decompose the remaining symbol with
 	the dyadic partition introduced in Section~\ref{sec:normal}:
 	\begin{equation*}
 		\begin{aligned}
 			m_N(\xi,\eta)=\chi_S(s,\xi,\eta)
 			&\left[1-\zeta\!\left(
 			\frac{\sqrt{|\xi-\eta|^2+|\eta|^2}}K\right)\right] \varphi\!\left(
 			\frac{\sqrt{|\xi-\eta|^2+|\eta|^2}}N\right).
 		\end{aligned}
 	\end{equation*}
 	For each fixed pair of spatial atoms, set
 	\[
 	w_N=\e^{-is\Delta}T_{m_N}(\e^{is\Delta}g_j,\e^{is\Delta}h_k).
 	\]
 	Throughout the frequency differentiations, the parameters
 	$s,j,k$, and hence $R$ and $K$, are kept fixed.
 	Only dyadic blocks with $N\ge K/2$ contribute,
 	and the symbol is supported where the joint frequency
 	is at least $K$.
 	Expanding the Fourier transforms of the spatial atoms,
 	we obtain
 	\begin{equation*}
 		\widehat w_N(\xi)=\frac1{(2\pi)^3}\iint
 		\mathscr K_{s,N}(\xi;y,z)g_j(y)h_k(z)\dd y\dd z,
 	\end{equation*}
 	where
 	\begin{align}
 		\mathscr K_{s,N}(\xi;y,z)
 		&=\int\e^{i\Theta(\xi,\eta;y,z)}m_N(\xi,\eta)\dd\eta,
 		\label{eq:high-kernel}\\
 		\Theta(\xi,\eta;y,z)
 		&=2s(\xi-\eta)\cdot\eta-y\cdot(\xi-\eta)-z\cdot\eta.
 		\notag
 	\end{align}
 	
 	By the support properties of the spatial atoms
 	and the definition $K=128R/s$, we have
 	\[
 	|y|+|z|\le4R=\frac{sK}{32}
 	\le\frac{s}{32}\sqrt{|\xi-\eta|^2+|\eta|^2}
 	\]
 	on the support of $m_N$. Combining this with
 	\eqref{eq:spacegap}, we obtain
 	\begin{equation}\label{eq:full-gap}
 		\begin{aligned}
 			|\nabla_\eta\Theta|
 			&=|2s(\xi-2\eta)+y-z|\\
 			&\ge\left(\frac12-\frac1{32}\right)
 			s\sqrt{|\xi-\eta|^2+|\eta|^2}\\
 			&\ge\frac{15}{64}sN.
 		\end{aligned}
 	\end{equation}
 	Thus the complete phase, which incorporates the
 	spatial locations of both atoms, is uniformly
 	nonstationary on the support of $m_N$.
 	
 	We next establish uniform nonstationary-phase estimates,
 	including derivatives with respect to the output frequency.
 	We rescale the frequency variables by setting
 	$\xi=N\widetilde\xi$ and $\eta=N\widetilde\eta$.
 	Then $\Theta=sN^2\widetilde\Theta$, where
 	\begin{equation*}		\widetilde\Theta
 		=2(\widetilde\xi-\widetilde\eta)\cdot\widetilde\eta
 		-\frac y{sN}\cdot(\widetilde\xi-\widetilde\eta)
 		-\frac z{sN}\cdot\widetilde\eta.
 	\end{equation*}
 	Defining the rescaled amplitude by
 	$a_{s,N}(\widetilde\xi,\widetilde\eta)
 	=m_N(N\widetilde\xi,N\widetilde\eta)$,
 	we can rewrite the kernel in \eqref{eq:high-kernel} as
 	\begin{equation*}		\mathscr K_{s,N}(N\widetilde\xi;y,z)
 		=N^2\int_{\mathbb R^2}
 		\e^{isN^2\widetilde\Theta}
 		a_{s,N}(\widetilde\xi,\widetilde\eta)
 		\dd\widetilde\eta.
 	\end{equation*}
 	
 	Since $N\ge K/2$ and $sK=128R$, we have
 	$sN\ge64R$. Together with $R\ge\sqrt s$
 	and \eqref{eq:full-gap}, this yields
 	\begin{equation}\label{eq:high-parameters}
 		sN^2\ge\frac{128^2}{4},\qquad
 		\frac{|y|+|z|}{sN}\le\frac1{16},\qquad
 		|\nabla_{\widetilde\eta}\widetilde\Theta|
 		\ge\frac{15}{64}.
 	\end{equation}
 	
 	The dyadic cutoff restricts the rescaled amplitude to
 	a fixed compact set away from the joint origin. The
 	low-frequency cutoff is constant there, and derivatives
 	of the moderate-to-high cutoff are nonzero only when
 	the dyadic and transition scales are comparable. Hence
 	the rescaled amplitude satisfies
 	\begin{equation}\label{eq:rescaled-amplitude}
 		\left|
 		\partial_{\widetilde\xi,\widetilde\eta}^{\beta}
 		a_{s,N}(\widetilde\xi,\widetilde\eta)
 		\right|
 		\le C_\beta
 	\end{equation}
 	for every multiindex $\beta$, where $C_\beta$
 	is independent of $s,j,k$, and $N$.
 	
 	We next differentiate the kernel with respect
 	to the output frequency before performing
 	integration by parts. Since the normalized
 	phase is affine in $\widetilde\xi$, we have
 	\[
 	\nabla_{\widetilde\xi}\widetilde\Theta
 	=2\widetilde\eta-\frac y{sN}.
 	\]
 	Thus, for $|\beta|\le2$, the chain rule and
 	Leibniz's formula give
 	\begin{equation}\label{eq:output-expansion}
 		\begin{aligned}
 			\partial_\xi^\beta
 			\mathscr K_{s,N}(\xi;y,z)
 			&=N^{-|\beta|}
 			\left.
 			\partial_{\widetilde\xi}^\beta
 			\mathscr K_{s,N}
 			(N\widetilde\xi;y,z)
 			\right|_{\widetilde\xi=\xi/N}\\
 			&=N^{2-|\beta|}
 			\sum_{\gamma\le\beta}
 			\binom{\beta}{\gamma}
 			(isN^2)^{|\gamma|}
 			\int_{\mathbb R^2}
 			\e^{isN^2\widetilde\Theta}
 			\left(2\widetilde\eta-\frac y{sN}\right)^\gamma
 			\partial_{\widetilde\xi}^{\beta-\gamma}
 			a_{s,N}(\widetilde\xi,\widetilde\eta)
 			\dd\widetilde\eta,
 		\end{aligned}
 	\end{equation}
 	where $\widetilde\xi=\xi/N$ in the integral
 	on the right-hand side.
 	
 	On a neighborhood of the amplitude support,
 	introduce the differential operator
 	\[
 	\mathcal L=
 	\frac{\nabla_{\widetilde\eta}\widetilde\Theta}
 	{i sN^2|\nabla_{\widetilde\eta}\widetilde\Theta|^2}
 	\cdot\nabla_{\widetilde\eta},
 	\qquad
 	\mathcal L\e^{isN^2\widetilde\Theta}
 	=\e^{isN^2\widetilde\Theta}.
 	\]
 	By the support properties of $m_N$ and
 	\eqref{eq:high-parameters}, we have
 	$|2\widetilde\eta-y/(sN)|\lesssim1$.
 	Moreover, after factoring out $(sN^2)^{-1}$,
 	the coefficients of $\mathcal L$ and their
 	derivatives of any fixed order are uniformly
 	bounded on the amplitude support.
 	
 	Applying integration by parts $M$ times
 	to each integral in \eqref{eq:output-expansion},
 	and using \eqref{eq:rescaled-amplitude}
 	together with $sN^2\ge1$, we obtain
 	\begin{equation}\label{eq:kernel-bound}
 		\begin{aligned}
 			|\partial_\xi^\beta\mathscr K_{s,N}(\xi;y,z)|
 			&\lesssim_{M,\beta}
 			N^{2-|\beta|}(sN^2)^{|\beta|-M}\\
 			&=N^2(sN)^{|\beta|}(sN^2)^{-M},
 			\qquad |\beta|\le2.
 		\end{aligned}
 	\end{equation}
 	
 	Finally, we use the kernel bound to estimate the
 	annular norm of each dyadic output and sum over $N$.
 	The joint-frequency localization restricts the output to $|\xi|\lesssim N$,
 	so the square root of its frequency-support area contributes a factor $N$.
 	By \eqref{eq:kernel-bound} and Minkowski's integral inequality,
 	\begin{equation*}
 		\begin{aligned}
 			\norm{\partial_\xi^\beta\widehat w_N}_2
 			&\le C\iint|g_j(y)||h_k(z)|
 			\norm{\partial_\xi^\beta\mathscr K_{s,N}(\cdot;y,z)}_2\dd y\dd z\\
 			&\lesssim_{M,\beta}
 			2^{j+k}\norm{g_j}_2\norm{h_k}_2
 			N^3(sN)^{|\beta|}(sN^2)^{-M}.
 		\end{aligned}
 	\end{equation*}
 	
 	Since $sN\ge64R\ge1$, we apply
 	Lemma~\ref{lem:moment} with $L=sN$ to obtain
 	\begin{equation*}
 		\begin{aligned}
 			\norm{w_N}_{\Z}
 			&\lesssim_M2^{j+k}\norm{g_j}_2\norm{h_k}_2
 			sN^4(sN^2)^{-M}\\
 			&=\frac{2^{j+k}}s\norm{g_j}_2\norm{h_k}_2
 			(sN^2)^{2-M}.
 		\end{aligned}
 	\end{equation*}
 	
 	Choosing $M=4$, we sum the dyadic estimates
 	using a geometric series:
 	\begin{equation*}
 		\sum_{N\ge K/2}(sN^2)^{-2}
 		\lesssim(sK^2)^{-2}\lesssim1,
 	\end{equation*}
 	where the last inequality follows from
 	\eqref{eq:space-basic-scales}, since
 	$sK^2=128^2R^2/s\ge128^2$.
 	Consequently,
 	\begin{equation}\label{eq:high-final}
 		\sum_{N\ge K/2}\norm{w_N}_{\Z}
 		\lesssim\frac{2^{j+k}}s\norm{g_j}_2\norm{h_k}_2.
 	\end{equation}

 	Combining \eqref{eq:moderate-final} and
 	\eqref{eq:high-final} yields
 	\eqref{eq:space-atom-goal}. Together with
 	\eqref{eq:res-final}, we obtain the required
 	estimates for each pair of spatial atoms.
 	
 	Finally, summing over $j,k\ge0$ gives
 	\[
 	\sum_{j,k\ge0}
 	\frac{2^{j+k}}s\|g_j\|_2\|h_k\|_2
 	=\frac1s\|g\|_{\mathcal Z}\|h\|_{\mathcal Z}.
 	\]
 	The estimates above give absolute convergence of the
 	spatial and high-frequency sums in $\mathcal Z$.
 	Truncation and density therefore extend the argument
 	from finite decompositions to general inputs. This proves
 	Lemma~\ref{lem:space}.

 \appendix
 
 \section{Addition of a gauge-invariant quadratic term}
 \label{app:gauge}
 
 In this appendix, we prove the extension stated
 in Remark~\ref{rem:gauge}. Consider
 \begin{equation}\label{eq:gauge-equation}
 	(i\partial_t+\Delta)u
 	=\lambda_1u^2+\lambda_2\overline u^{\,2}
 	+\lambda_3|u|u,
 	\qquad \lambda_3\in\mathbb C,
 \end{equation}
 with the same initial data as in
 Theorem~\ref{thm:main}.
 The additional gauge-invariant term can be
 estimated directly in the profile space $\mathcal Z$.
 
 We first establish the following estimate.
 
 \begin{lemma}\label{lem:gauge-profile}
 	For every $s\ge1$ and $g\in\mathcal Z$,
 	\begin{equation}\label{eq:gauge-Z}
 		\left\|
 		e^{-is\Delta}
 		\bigl(|e^{is\Delta}g|e^{is\Delta}g\bigr)
 		\right\|_{\mathcal Z}
 		\lesssim
 		\|e^{is\Delta}g\|_\infty
 		\|g\|_{\mathcal Z}
 		\lesssim
 		s^{-1}\|g\|_{\mathcal Z}^2.
 	\end{equation}
 \end{lemma}

\begin{proof}
	For fixed $s\ge1$, let
	\[
	h(x)=e^{i|x|^2/(4s)}g(x).
	\]
	The explicit formula for the free Schr\"odinger
	group gives
	\[
	e^{is\Delta}g(x)
	=\frac1{2is}e^{i|x|^2/(4s)}
	\widehat h\left(\frac{x}{2s}\right).
	\]
	Moreover, by the gauge covariance of $z\mapsto |z|z$,
	we obtain
	\[
	e^{-is\Delta}
	\bigl(|e^{is\Delta}g|e^{is\Delta}g\bigr)
	=
	\frac1{2s}e^{-i|x|^2/(4s)}
	\mathcal F^{-1}\bigl(|\widehat h|\widehat h\bigr).
	\]
	
 Thus the equivalence
	$\|w\|_{\mathcal Z}
	\approx\|\widehat w\|_{B^1_{2,1}}$
 impies
	\[
	\begin{aligned}
		\left\|
		e^{-is\Delta}
		\bigl(|e^{is\Delta}g|e^{is\Delta}g\bigr)
		\right\|_{\mathcal Z}	&\lesssim
		s^{-1}\||\widehat h|\widehat h\|_{B^1_{2,1}}\lesssim
		s^{-1}\|\widehat h\|_\infty
		\|\widehat h\|_{B^1_{2,1}}\lesssim
		\|e^{is\Delta}g\|_\infty
		\|g\|_{\mathcal Z}.
	\end{aligned}
	\]
	The remaining bound in \eqref{eq:gauge-Z}
	follows from \eqref{eq:basic-intro}.
\end{proof}

 We now prove the extension of
 Theorem~\ref{thm:main}.
 
 \begin{proof}[Proof of the extension]
 	We retain the notation $f(t)$, $A(t)$, and $X_T$
 	from Section~\ref{sec:reduction}, and set
 	\[
 	\Lambda=|\lambda_1|+|\lambda_2|+|\lambda_3|.
 	\]
 	
 	We apply the normal-form decomposition
 	of Section~\ref{sec:normal} only to
 	$u^2$ and $\overline u^{\,2}$.
 	The additional term remains in the Duhamel
 	integral and contributes
 	\[
 	-i\lambda_3\int_1^t
 	e^{-is\Delta}\bigl(|u(s)|u(s)\bigr)\,ds.
 	\]
 	By Lemma~\ref{lem:gauge-profile},
 	\begin{equation}\label{eq:gauge-duhamel}
 		\begin{aligned}
 			&\left\|
 			\lambda_3\int_1^t
 			e^{-is\Delta}\bigl(|u(s)|u(s)\bigr)\,ds
 			\right\|_{\mathcal Z}\lesssim
 			|\lambda_3|\int_1^t\frac{A(s)^2}{s}\,ds.
 		\end{aligned}
 	\end{equation}
 	
 	For the cubic normal-form remainders,
 	the modified profile equation gives
 	\[
 	\|e^{i\alpha s\Delta}\partial_sf_\alpha\|_\infty
 	\lesssim
 	\Lambda s^{-2}A(s)^2,
 	\qquad \alpha=\pm1.
 	\]
 	The one-sided estimate
 	\eqref{eq:onesided-normal}
 	therefore applies as before, with constants
 	depending also on $\lambda_3$.
 	
 	Combining these bounds with the estimates
 	in Section~\ref{sec:reduction}, we obtain
 	\begin{equation}\label{eq:gauge-apriori}
 		X_T\le C_0\varepsilon
 		+C_\Lambda(1+\log T)X_T^2
 		+C_\Lambda(1+\log T)^2X_T^3.
 	\end{equation}
 	The standard $H^2$ energy and weighted
 	persistence estimates also extend to
 	$|u|u$, using its $C^{1,1}$ regularity.
 	Consequently, the bootstrap and continuation
 	arguments in Section~\ref{sec:reduction}
 	apply with modified constants and yield
 	all the conclusions of Theorem~\ref{thm:main}
 	for \eqref{eq:gauge-equation}.
 \end{proof}

 	\Needspace{8\baselineskip}
 	\bigskip
 	\noindent\textbf{Acknowledgements:} X.~Liu was supported by NSFC Grant 12501314.

 \bigskip
 \noindent\textbf{Declaration of AI use.}
 ChatGPT was used for research-related discussions
 and to assist with language editing, manuscript
 organization, LaTeX preparation, bibliographic
 verification, and consistency checks.
The author critically reviewed and revised
all content and takes full responsibility for the manuscript.

 	\bigskip
 	\noindent\textsc{School of Mathematics, Hangzhou Normal University,
 		Hangzhou 311121, China}
 	
 	\noindent\textit{Email address:} \texttt{liuxuan95@hznu.edu.cn}.
 \end{document}